\documentclass[11pt, a4paper, reqno, english]{amsart}

\usepackage{hyperref}
\usepackage{mathrsfs}
\usepackage[all]{xy}

\newcommand\kbar{{\overline{k}}}
\newcommand\Xbar{{\overline{X}}}
\newcommand\Ubar{{\overline{U}}}
\newcommand\Ybar{{\overline{Y}}}
\newcommand\That{{\widehat{T}}}
\newcommand\Mhat{{\widehat{M}}}
\newcommand\tX{{\widetilde{X}}}
\newcommand\ttt{\mathbf{t}}
\newcommand\xx{\mathbf{x}}
\newcommand\yy{\mathbf{y}}
\newcommand\zz{\mathbf{z}}
\newcommand\AAA{\mathbb{A}}
\newcommand\ZZ{\mathbb{Z}}
\newcommand\QQ{\mathbb{Q}}
\newcommand\RR{\mathbb{R}}

\newcommand\GG{\mathbb{G}}
\newcommand\GmZ{\GG_{\mathrm{m}}}
\newcommand\Gm[1]{\GG_{\mathrm{m},#1}}
\newcommand\sm{\mathrm{sm}}
\newcommand\T{\mathcal{T}}

\newcommand\XX{\frak{X}}

\newcommand\YY{\frak{Y}}
\newcommand\OO{\mathcal{O}}

\DeclareMathOperator\Pic{Pic}
\DeclareMathOperator\Div{Div}

\DeclareMathOperator\Gal{Gal}
\DeclareMathOperator\Hom{Hom}
\DeclareMathOperator\Br{Br}
\DeclareMathOperator\disc{disc}
\DeclareMathOperator\inv{inv}

\DeclareMathOperator\vol{vol}
\DeclareMathOperator\sgn{sign}

\newtheorem{theorem}{Theorem}
\newtheorem{question}[theorem]{Question}
\newtheorem{lemma}[theorem]{Lemma}
\newtheorem{prop}[theorem]{Proposition}
\newtheorem{cor}[theorem]{Corollary}
\theoremstyle{definition}
\newtheorem{definition}[theorem]{Definition}
\newtheorem{example}[theorem]{Example}
\newtheorem*{ack}{Acknowledgements}
\newtheorem*{terminology}{Terminology}
\newtheorem{remark}[theorem]{Remark}
\numberwithin{equation}{section}
\numberwithin{theorem}{section}

\begin{document}

\title{Strong approximation and descent}

\author{Ulrich Derenthal}

\address{Institut f\"ur Algebra, Zahlentheorie und Diskrete
  Mathematik, Leibniz Universit\"at Hannover,
  Welfengarten 1, 30167 Hannover, Germany}

\email{derenthal@math.uni-hannover.de}

\author{Dasheng Wei}

\address{Academy of Mathematics and System Science, CAS, Beijing 100190,
  P.\ R.\ China}

\email{dshwei@amss.ac.cn}

\date{November 26, 2014}

\begin{abstract}
  We introduce descent methods to the study of strong approximation on
  algebraic varieties. We apply them to two classes of varieties
  defined by $P(\ttt)=N_{K/k}(\zz)$: firstly for quartic extensions of
  number fields $K/k$ and quadratic polynomials $P(\ttt)$ in one
  variable, and secondly for $k=\QQ$, an arbitrary number field $K$
  and $P(\ttt)$ a product of linear polynomials over $\QQ$ in at least
  two variables. Finally, we illustrate that a certain unboundedness
  condition at archimedean places is necessary for strong
  approximation.
\end{abstract}


%
%

\maketitle

\setcounter{tocdepth}{1}
\tableofcontents

\section{Introduction}

Let $k$ be a number field. We study strong approximation with Brauer--Manin
obstruction for two families of algebraic varieties $X \subset \AAA^{n+s}_k$
defined by equations of the form
\begin{equation}\label{eq:norm}
  P(\ttt) = N_{K/k}(\zz),
\end{equation}
where $P(\ttt) \in k[t_1, \dots, t_s]$ is a polynomial in $s$ variables over
$k$ and $N_{K/k}$ is a norm form for an extension $K/k$ of degree $n$.

In our proofs, we use descent to reduce the problem to strong approximation on
their universal torsors. While descent has been applied frequently to weak
approximation, the precise formulation of our descent lemma seems to be
crucial for its first applications to strong approximation.

To prove strong approximation on universal torsors, we reduce it to
quadrics by the fibration method in one case, and we use a
generalization of a result of Browning and Matthiesen, based on
methods of Green, Tao, and Ziegler from additive combinatorics, in the
other case.

In our results, we encounter an unboundedness condition at the archime\-dean
places. We give a counterexample to strong approximation that shows that such
a condition is generally necessary.

\subsection*{Background}

Let $X$ be an algebraic variety defined over a number field $k$. We say that
strong approximation holds for $X$ off a finite set $S$ of places of $k$ if
the image of the set $X(k)$ of rational points on $X$ is dense in the space
$X(\AAA_k^S)$ of adelic points on $X$ outside $S$. Strong approximation for $X$
off $S$ implies the Hasse principle for $S$-integral points on any
$S$-integral model of $X$.

For a proper variety $X$, strong approximation off $S=\emptyset$ is the same
as weak approximation. For an affine variety $X$, however, studying strong
approximation and the Hasse principle for integral points on integral models
seems to be generally much harder than studying weak approximation and the
Hasse principle for rational points on proper models of $X$.

Failures of weak approximation and the Hasse principle for rational points on
proper varieties are often explained by Brauer--Manin obstructions, introduced
by Manin \cite{MR0427322}. Only recently this was generalized to strong
approximation by Colliot-Th\'el\`ene and Xu \cite{MR2501421}.

For algebraic groups and their homogeneous spaces, weak and strong
approximation and the Hasse principle have been widely studied. For
certain simply connected semisimple groups and their principal
homogeneous spaces, we have the classical work of Kneser, Harder,
Platonov and others. For certain homogeneous spaces of connected
algebraic groups with connected or abelian stabilizers, see for
example \cite{MR1390687, MR2501421, MR2429455,MR2783137, MR2928335, MR3008912,MR2989976} and the
references therein for weak and strong approximation with
Brauer--Manin obstruction. This includes varieties defined by
(\ref{eq:norm}) for constant $P$ and arbitrary $K/k$, and for $s=1$
with $\deg(P(t))=2$ and $[K:k]=2$.

Much less is known for more general varieties that are not homogeneous
spaces of algebraic groups. For weak approximation with Brauer--Manin
obstruction, let us mention the now classical example of Ch\^atelet
surfaces \cite{MR870307, MR876222}, which are actually smooth proper
models of certain varieties defined by (\ref{eq:norm}). See
\cite{MR2989431, arXiv:1202.3567, arXiv:1202.4115} and the references therein for weak
approximation results for several further classes of varieties defined
by (\ref{eq:norm}).

A recent breakthrough is the introduction of deep results from additive
combinatorics due to Green--Tao--Ziegler and Matthiesen to deduce weak
approximation for varieties of the form (\ref{eq:norm}) when $P(t)$ is
a product of arbitrarily many linear polynomials over $\QQ$
\cite{MR3194818, arXiv:1304.3333, arXiv:1307.7641}.

For strong approximation with Brauer--Manin obstruction for affine varieties
defined by equations of the form
\begin{equation*}
  P(t) = q(z_1,z_2,z_3),
\end{equation*}
see \cite{MR3007293}. For more general fibrations over $\AAA^1_k$ with
\emph{split} (e.g., geometrically integral) fibers, see
\cite{arXiv:1209.0717}.

Colliot-Th\'el\`ene and Harari \cite[p.~4]{arXiv:1209.0717} ask for
the integral Hasse principle and strong approximation for the
equation~(\ref{eq:norm}) with $s=1$, a separable polynomial $P(t)$ of
degree at least $3$, and $[K:k]=2$. They say that this is out of
reach of the current techniques because of the following two
essential difficulties: on the one hand, $\Br(X_t)/\Br(k)$ is infinite
for each smooth $k$-fiber $X_t$ for the natural fibration $X \to \AAA^1_k$
via projection to the $t$-coordinate, and on the other hand, the fibers
over the roots of $P(t)$ are not split. Note that the same
difficulties occur for $P(t)$ of degree at least $2$, and $[K:k] \ge
3$.

Counterexamples to strong approximation explained by Brauer--Manin
obstructions can be found in \cite{MR2471948, MR3106738}. See also
\cite{MR2975237} for computations of Brauer--Manin obstructions for integral
points on certain cubic surfaces.

\subsection*{Our results}

We obtain the first strong approximation results for varieties defined
by (\ref{eq:norm}). Here, both of the difficulties occur that were
mentioned in \cite[p.~4]{arXiv:1209.0717}. In our main theorems, we
consider two families of such varieties. Our notation is mostly
standard; see the end of the introduction for a reminder.

\begin{theorem}\label{thm:main}
  Let $K/k$ be an extension of number fields.  Assume that
  $P(t)=c(t^2-a) \in k[t]$ is an irreducible quadratic polynomial, and
  $[K:k]=4$ with $\sqrt{a} \in K$. Let $X \subset \AAA^5_k$ be defined
  by (\ref{eq:norm}) with $s=1$.

  Assume that there is an archimedean place $v_0$ such that
  $p(X(k_{v_0}))$ is not bounded, where $p:X \to \AAA^1_k$ is the
  projection to the $t$-coordinate.  Then strong approximation with
  Brauer--Manin obstruction holds for $X$ off $v_0$.
\end{theorem}

We also compute $\Br(X)$ for $X$ as in Theorem~\ref{thm:main}, and in certain
cases we deduce the Hasse principle for integral points and strong
approximation (without Brauer--Manin obstructions), see
Corollaries~\ref{cor:strong_approximation} and \ref{cor:hasse_principle}.

\begin{theorem}\label{thm:totally_split}
  Let $k=\QQ$. For $s\ge 2$, let $P(\ttt) \in \QQ[t_1, \dots, t_s]$ be a
  product of pairwise proportional or affinely independent linear polynomials
  over $\QQ$.  Let $K/\QQ$ be an arbitrary extension of number fields of
  degree $n$. Let $X \subset \AAA^{n+s}_\QQ$ be the affine variety defined by
  (\ref{eq:norm}).

  Let $C$ be the union of the connected components of $p(X^\sm(\RR))$
  that contain balls of arbitrarily large radius, where $p:X \to
  \AAA^s_k$ is the projection to the $\ttt$-coordinates. Then $X^\sm(k)$
  is dense in the image of $(p^{-1}(C) \times
  X^\sm(\AAA_k^f))^{\Br_1(X^\sm)}$ in $X(\AAA_k^f)$.

  If $K$ is not totally imaginary or if the factors of $P(\ttt)$ are linear
  forms, then $C=p(X^\sm(\RR))$, and $X$ satisfies smooth strong approximation
  with algebraic Brauer--Manin obstruction off $\infty$ (see
  Definition~\ref{def:smooth_sa_hp}).
\end{theorem}

Furthermore, we can deduce a smooth Hasse principle with Brauer--Manin
obstruction for integral points on $X$ as in
Theorem~\ref{thm:totally_split}. See
Corollary~\ref{cor:smooth_hasse_principle} for details.

\subsection*{Techniques}

There are two fundamental techniques to reduce the study of weak approximation
and the Hasse principle for rational points (possibly with Brauer--Manin
obstruction) for one class of varieties to the same questions for other
varieties, namely the \emph{fibration method} and the \emph{descent
  method}. One may ask whether these techniques can also be applied in the
context of strong approximation.

The fibration method typically applies to fibrations $f: X \to Y$ where weak
approximation or the Hasse principle is known both for the fibers of $f$ and
the base $Y$. An example is the deduction of the Hasse principle for quadratic
forms in five variables from the four-variable-case. See \cite{MR870307,
  MR876222} for more involved applications, and \cite{MR1284820} for its
combination with Brauer--Manin obstructions.

Generalizing the fibration method to strong approximation is achieved in some
generality in \cite{MR3007293, arXiv:1209.0717}.

The descent method reduces the study of weak approximation and the Hasse
principle with (algebraic) Brauer--Manin obstruction to their study on torsors
under tori over the original variety (for example \emph{universal
  torsors}). This is expected to simplify the task since typically no
(algebraic) Brauer--Manin obstruction occurs on universal torsors. See
\cite{MR592151, MR870307, MR876222} for applications of
descent to weak approximation.

To our knowledge, Theorems~\ref{thm:main} and \ref{thm:totally_split} are the
first applications of the descent method to strong approximation. For this, an
important auxiliary result is the descent lemma presented in
Section~\ref{sec:descent_lemma}. A subtle point was to find the right
formulation that makes it applicable in practice. Then the proof of this lemma is an
easy application of descent theory as introduced by Colliot-Th\'el\`ene and
Sansuc, see \cite{MR89f:11082}.

Our descent lemma reduces strong approximation with algebraic Brauer--Manin
obstruction on the original variety to strong approximation on auxiliary
varieties containing open subsets of universal torsors.

For $X$ as in Theorem~\ref{thm:main}, we have determined a local description of
universal torsors in \cite{arXiv:1202.3567}. Here, we observe that these are
essentially fibrations over $\AAA^4_k$ whose fibers are smooth $3$-dimensional
quadrics, so that an application of the fibration method yields the result.

For $X$ as in Theorem~\ref{thm:totally_split}, we show that the
universal torsors are essentially varieties for which weak
approximation was proved in recent work of Browning and Matthiesen
\cite{arXiv:1307.7641} (on $X$ defined over $k=\QQ$ by (\ref{eq:norm})
in cases where $P(t)$ is totally split over $\QQ$, with $s=1$; note
that we cannot treat the case $s=1$ in
Theorem~\ref{thm:totally_split}, see Remark~\ref{remark:s_ge_2} for
the reason), based on results from additive combinatorics by
Green--Tao--Ziegler and Matthiesen.  From a generalization of the key
technical result \cite[Theorem~5.2]{arXiv:1307.7641} of Browning and
Matthiesen from linear forms to linear polynomials, we deduce that the
varieties containing our universal torsors also satisfy strong
approximation.

The descent method may lead to further results on strong approximation
for varieties defined by (\ref{eq:norm}) when combined with other
analytic techniques. For example, it would be interesting to combine
it with sieve methods (see \cite{MR1961196, arXiv:1202.3567} and
\cite{arXiv:1310.6158} for results on weak approximation) and the
circle method (see \cite{MR1961196} and \cite{MR3188633} for results
on weak approximation).

\subsection*{An unboundedness condition}

We observe that Theorems~\ref{thm:main} and \ref{thm:totally_split}
include unboundedness conditions at an archimedean place. This has
some resemblence with conditions at archimedean places appearing in
strong approximation results on homogeneous spaces of algebraic
groups, e.g., \cite[Theorem~7.12]{MR1278263}.

Example~\ref{exa:bounded} shows that strong approximation with
Brauer--Manin obstruction off $\infty$ does not hold for the variety
$X \subset \AAA^3_\QQ$ defined by
\begin{equation*}
  t(t-2)(t-10)=x^2+y^2,
\end{equation*}
which is an example of (\ref{eq:norm}). Here, $X(\RR)$ has a bounded and an
unbounded connected component.

This counterexample and Theorems~\ref{thm:main} and
\ref{thm:totally_split} lead us to the expectation that only the
following version of smooth strong approximation can be true for
varieties defined by (\ref{eq:norm}).  See also
\cite{arXiv:1209.0717}.

\begin{question}\label{qu:strong}
  Let $K/k$ be an extension of number fields of degree $n$, let $P(\ttt) \in
  k[t_1, \dots, t_s]$ be a non-constant polynomial. Let $X \subset
  \AAA^{n+s}_k$ be the affine variety over $k$ defined by
  (\ref{eq:norm}). Let $v_0$ be an archimedean place.

  Let $p : X \to \AAA^s_k$ be the projection to the
  $\ttt$-coordinates. Let $C$ be the union of the connected components
  of $p(X^\sm(k_{v_0}))$ that contain balls of
  arbitrarily large radius.

  Is $X^\sm(k)$ dense in the image of $(p^{-1}(C) \times X^\sm(\AAA_k^{\{v_0\}}))^{\Br(X^\sm)}$ in
  $X(\AAA_k^{\{v_0\}}))$ with respect to the adelic topology?
\end{question}

Theorems~\ref{thm:main} and \ref{thm:totally_split} give an
affirmative answer to this question for two families of
varieties.

\begin{terminology}
  For a field $k$ of characteristic $0$, fix an algebraic closure
  $\kbar$, and let $\Gamma_k$ be the absolute Galois group. Let
  $\Br(k)$ be the Brauer group of $k$. For a scheme $X$ over $k$, let
  $X^\sm$ be its smooth locus, and let $\Xbar:= X \times_k \kbar$. Let
  $\Br(X):=H^2_\text{\'et}(X,\GmZ)$ be the cohomological Brauer group,
  $\Br_0(X)$ its subgroup of constant elements, namely the image of
  the natural map $\Br(k) \to \Br(X)$, and $\Br_1(X)$ its algebraic
  Brauer group, namely the kernel of the natural map $\Br(X) \to
  \Br(\Xbar)$.

  Now let $k$ be a number field. Then $\Omega_k$ denotes the set of
  places of $k$, and $\infty_k$ denotes its subset of archimedean
  places. We write $v < \infty_k$ for $v \in \Omega_k \setminus
  \infty_k$. The ring of integers in $k$ is denoted by $\OO_k$. For $v
  \in \Omega_k$, let $k_v$ be the completion of $k$ at the place $v$,
  and let $\OO_v$ be the ring of integers in $k_v$.  The adele ring
  with its usual adelic topology is denoted by $\AAA_k$.  For a finite
  subset $S \subset \Omega_k$, let $\OO_S$ be the ring of $S$-integers
  of $k$, and let $\AAA_k^S \subset \prod_{v \in \Omega_k \setminus S}
  k_v$ be the adeles without $v$-component for all $v \in S$, which
  also comes with a natural adelic topology. In particular, we write
  $\AAA_k^f:=\AAA_k^{\infty_k}$ for the adeles without archimedean
  components.

  Let $X$ be a geometrically integral variety over a number field $k$,
  and let $S \subset \Omega_k$ be a finite set of places. We say that
  \emph{strong approximation holds for $X$ off $S$} if $X(k)$ is dense
  in the image of $X(\AAA_k)$ in $X(\AAA_k^S)$.

  One says that \emph{strong approximation with (algebraic)
    Brauer--Manin obstruction holds for $X$ off $S$} if $X(k)$ is
  dense in the image of $X(\AAA_k)^{\Br(X)}$ (resp.\
  $X(\AAA_k)^{\Br_1(X)}$) in $X(\AAA_k^S)$ under the natural projection
  (see \cite[Definition~2.4]{MR3007293}). Here, $X(\AAA_k)^B$ is the
  set of all adelic points $(x_v) \in X(\AAA_k)$ satisfying $\sum_{v
    \in \Omega_k} \inv_v(\beta(x_v)) = 0$ for all $\beta$ in a subset $B
  \subset \Br(X)$, where $\inv_v: \Br(k_v) \to \QQ/\ZZ$ is the
  invariant map from local class field theory.
\end{terminology}

\begin{ack}
  The first author was supported by grants DE 1646/2-1 and DE 1646/3-1 of the
  Deutsche Forschungsgemeinschaft. The second author is supported by National
  Key Basic Research Program of China (Grant No. 2013CB834202) and
  National Natural Science Foundation of China (Grant Nos. 11371210 and
  11321101).

  We thank J.-L.~Colliot-Th\'el\`ene and L.~Matthiesen for useful
  remarks and discussions, in particular during the first author's
  stay at the IH\'ES (September/October 2013), whose support and
  hospitality is gratefully acknowledged. We are grateful to the
  anonymous referees for their helpful comments.
\end{ack}

\section{Strong approximation on singular varieties}\label{sec:smooth_SA}

Let $k$ be a number field. For smooth varieties $X$ over $k$, it is
interesting to study strong approximation because we can derive the
existence of integral points on any integral model of $X$.

For singular varieties $X$ over $k$, the implicit function theorem may
fail, and hence we cannot hope to prove strong approximation on $X$.
On the other hand, integral models of $X^\sm$ often have far less
integral points than integral models of $X$, hence strong
approximation on $X^\sm$ is generally too much to ask for.

Instead, we introduce the following notion of \emph{smooth strong
  approximation} on $X$ which is suitable to determine the existence of
integral points on any integral model of the singular variety $X$ (see
Remark~\ref{rem:smooth_sa_hp} and Corollary~\ref{cor:smooth_hasse_principle}).

\begin{definition}\label{def:smooth_sa_hp}
  Let $X$ be a geometrically integral variety over a number field $k$. Let $S$ be a finite set of places of
  $k$. We say that $X$ satisfies \emph{smooth strong approximation off
    $S$} if $X^\sm(k)$ is dense in the image of $X^\sm(\AAA_k)$ in
  $X(\AAA_k^S)$ under the natural projection.

  Analogously, we say that $X$ satisfies \emph{smooth strong approximation with
    (algebraic) Brauer--Manin obstruction off $S$} if $X^\sm(k)$ is dense in
  the image of $X^\sm(\AAA_k)^{\Br(X^\sm)}$
  (resp. $X^\sm(\AAA_k)^{\Br_1(X^\sm)}$) in $X(\AAA_k^S)$.

  We say that $X$ satisfies the \emph{smooth integral Hasse principle
    (with (algebraic) Brauer--Manin obstruction)} if the following holds
  for any integral model $\XX$ of $X$: If
  \begin{equation}\label{eq:smooth_hp_brauer_set}
    X^\sm(\AAA_k)^B \cap \left(\prod_{v \in \infty_k}
      X(k_v) \times \prod_{v < \infty_k} \XX(\OO_v)\right)
  \end{equation} is non-empty (where $B$ is $\emptyset$ resp.\ $\Br(X^\sm)$ resp.\
  $\Br_1(X^\sm)$), then there are integral points on $\XX$.
\end{definition}

For singular $X$, note that smooth strong approximation on $X$ off $S$
is not the same as strong approximation on $X^\sm$: the latter means
that $X^\sm(k)$ is dense in the image of $X^\sm(\AAA_k)$ in
$X^\sm(\AAA_k^S)$, whose adelic topology is stronger than the topology
induced by $X(\AAA_k^S)$.

\begin{remark}\label{rem:smooth_sa_hp}
  Assume that $X$ is a variety over a number field $k$ satisfying
  smooth strong approximation (with (algebraic) Brauer--Manin
  obstruction) off $\infty_k$. Then the smooth integral Hasse
  principle (with (algebraic) Brauer--Manin obstruction) holds on any integral model of $X$.

  Indeed, for any integral model $\XX$, we have $(q_v)$ in the set
  (\ref{eq:smooth_hp_brauer_set}). An open neighborhood of $(q_v)_{v <
    \infty_k} \in X(\AAA_k^f)$ is given by $\prod_{v < \infty_k}
  \XX(\OO_v)$. Then smooth strong approximation off $\infty_k$ gives a
  $q \in X^\sm(k)$ which lies in this open neighborhood. This
  ensures $q \in \XX(\OO_k)$.
\end{remark}

Next, we compare smooth strong approximation to central strong approximation
\cite[\S 8]{MR3007293}. Recall that it can be defined as follows:

\begin{definition}[Colliot-Th\'el\`ene, Xu]
  Let $X$ be a geometrically integral variety over a number field
  $k$. Let $S$ be a finite set of places of $k$. Let $f: \tX \to X$ be
  a resolution of singularities for $X$. The following two definitions
  do not depend on the choice of $\tX$.

  One says that \emph{central strong approximation holds for $X$ off $S$} if
  the diagonal image of $X^\sm(k)$ is dense in the natural image of
  $\tX(\AAA_k)$ in $X(\AAA_k^S)$.

  Assume that $\Br(\tX)/\Br_0(\tX)$ (resp. $\Br_1(\tX)/\Br_0(\tX)$) is
  finite. One says that \emph{central strong approximation with (algebraic)
    Brauer--Manin obstruction holds for $X$ off $S$} if the diagonal image of
  $X^\sm(k)$ is dense in the natural image of $\tX(\AAA_k)^{\Br(\widetilde X)}$
  (resp. $\tX(\AAA_k)^{\Br_1(\widetilde X)}$) in $X(\AAA_k^S)$.
\end{definition}

Note that the finiteness assumption on the Brauer group is generally
necessary to ensure that the definition of central strong
approximation with Brauer--Manin obstruction is independent of the
choice of $\tX$, but that no such finiteness condition is needed in
our definition of smooth strong approximation.

\begin{remark}\label{rem:comparison}
  Let $X$ be a geometrically integral variety over a number field $k$, and let
  $S$ be a finite set of places of $k$. Then smooth strong approximation on
  $X$ off $S$ is equivalent to central strong approximation on $X$ off $S$.

  Assume that $\Br(X^\sm)/\Br_0(X^\sm)$ (resp. $\Br_1(X^\sm)/\Br_0(X^\sm)$) is
  finite. Then smooth strong approximation with (algebraic) Brauer--Manin
  obstruction on $X$ off $S$ is equivalent to central strong approximation
  with (algebraic) Brauer--Manin obstruction on $X$ off $S$.

  This is proved via the version \cite[Th\'eor\`eme~1.4]{MR2011747} of
  Harari's formal lemma \cite[Corollaire~2.6.1]{MR1284820}.
\end{remark}

\section{A descent lemma}\label{sec:descent_lemma}

The following lemma, based on descent theory of Colliot-Th\'el\`ene and Sansuc
\cite{MR89f:11082}, is central in our proofs of strong approximation in the
following sections.

\begin{lemma}\label{lem:passage}
  Let $k$ be a number field. Let $X$ be an integral variety over $k$ with
  $\kbar[X^\sm]^\times = \kbar^\times$ and $\Pic(\Xbar^\sm)$ of finite
  type, and let $U$ be a dense open subset of $X^\sm$.

  Let $S$ be a finite subset of $\Omega_k$.  For any universal torsor
  $f: \T \to X^\sm$, assume that its restriction $\T_U:= \T
  \times_{X^\sm} U$ is geometrically integral, and that there is a
  commutative diagram
  \begin{equation*}
    \xymatrix{\T_U \ar@{^{(}->}[r]^i \ar[d]_{f_{|\T_U}} & Y \ar[d]_g\\
      U \ar@{^{(}->}[r] & X}
  \end{equation*}
  where $Y$ is a variety over $k$ satisfying smooth strong
  approximation off $S$, and $i: \T_U \to Y$ is an open immersion.

  Then $X$ satisfies smooth strong approximation with algebraic
  Brauer--Manin obstruction off $S$.
\end{lemma}

\begin{proof}
  We must find a rational point $p \in X^\sm(k)$ which approximates the projection of a
  given $(p_v) \in X^\sm(\AAA_k)^{\Br_1(X^\sm)}$ in $X(\AAA_k^S)$.

  By descent theory (see \cite{MR89f:11082},
  \cite[Theorem~3]{MR1666779}), there is a universal torsor $f: \T \to
  X^\sm$ and $(r_v) \in \T(\AAA_k)$ such that $(f(r_v)) = (p_v)$.

  Since $\T_U$ is geometrically integral, any integral model $\frak{T}_U$
  of $\T_U$ satisfies $\frak{T}_U(\OO_v) \ne \emptyset$ for almost all $v
  < \infty_k$. This implies that $\T_U(\AAA_k)$ is dense in $\T(\AAA_k)$ since $\T$
  is smooth and $\T_U$ is open and dense in it.  Therefore, we can find
  $(r_v') \in \T_U(\AAA_k)$ very close to $(r_v)$ in $\T(\AAA_k)$.

  By assumption, we have $\T_U \subset Y^\sm \subset Y$ for some $Y$
  satisfying smooth strong approximation off $S$. Hence we obtain a
  point $r \in Y^\sm(k)$ arbitrarily close to the projection of
  $(r_v') \in \T_U(\AAA_k) \subset Y^\sm(\AAA_k)$ in $Y(\AAA_k^S)$;
  for $r$ close enough, we have $r \in \T_U(k)$.

  Let $p := f(r) \in U(k)$. Since $(r_v')$ is very close to $(r_v)$ in
  $\T(\AAA_k)$, we know that $(p_v'):=(f(r_v'))$ is very close to
  $(p_v)$ in $X^\sm(\AAA_k)$. Since $X^\sm$ is open in $X$, also $(p_v')$ is
  very close to $(p_v)$ with respect to the adelic topology of
  $X(\AAA_k)$. Furthermore, since $r$ is very close to the projection
  of $(r_v')$ in $Y(\AAA_k^S)$, we have $p$ very close to the
  projection of $(p_v') = (g(r_v'))$ in $X(\AAA_k^S)$. Hence $p$ is
  very close to the projection of $(p_v)$ in $X(\AAA_k^S)$.
\end{proof}

\section{Quadratic polynomials represented by quartic norms}

In this section, we apply our Descent Lemma~\ref{lem:passage} to prove
Theorem~\ref{thm:main}. The main work lies in
Proposition~\ref{prop:strong_approximation_torsor} proving strong
approximation for the varieties $Y \subset \AAA^8_k$ defined
by~(\ref{eq:universal_torsors}), which essentially are the universal torsors,
using their fibration over $\AAA^4_k$ whose fibers are three-dimensional
quadrics.

Additionally, we compute the Brauer group of $X$ in
Proposition~\ref{prop:brauer} and obtain results on the integral Hasse
principle.

\begin{proof}[Proof of Theorem~\ref{thm:main}]
  Let $X$ be as in Theorem~\ref{thm:main}. Since $P(t)$ is separable, $X$ is
  smooth over $k$. Let $U := X \cap \{P(t) \ne 0\}$. By the local description
  of its universal torsors \cite[Proposition~2]{arXiv:1202.3567} (see also
  \cite[proof of Proposition~3]{arXiv:1202.3567}), we have
\begin{equation*}
  \T_U \subset \AAA^1_k \times R_{K/k}(\Gm{K})^2
\end{equation*}
defined by
\begin{equation}\label{eq:original_universal_torsors}
  t-\sqrt a=\rho\cdot N_{K/L}(\xx)\cdot \sigma(N_{K/L}(\yy)) \ne 0
\end{equation}
for $k \subset L:=k(\sqrt a) \subset K$, some $(\rho, \xi) \in L^\times
\times K^\times$ satisfying $c N_{L/k}(\rho) = N_{K/k}(\xi)$, and
$\sigma \in \Gamma_k$ with $\sigma(\sqrt a)=-\sqrt a$.  The
restriction of $f : \T \to X$ to $\T_U$ is given by $(t,\xx,\yy)
\mapsto (t,\xi\xx\yy)$.

We write $K=L(\sqrt{u+v\sqrt a})$ with $u,v \in k$. For
$\xx=(x_1+x_2\sqrt a)+(x_3+x_4\sqrt a)\sqrt{u+v\sqrt a} \in K$ and
$\yy \in K$, we have
\begin{align*}
  N_{K/L}(\xx) &= (x_1+x_2\sqrt a)^2-(x_3+x_4\sqrt a)^2(u+v\sqrt a)\\
  &=g_0(\xx)+g_1(\xx)\sqrt a,\\
  \rho \cdot \sigma(N_{K/L}(\yy)) &= \lambda_0(\yy)+\lambda_1(\yy)\sqrt a,
\end{align*}
for quadratic forms
\begin{align*}
  g_0(\xx) &= x_1^2+ax_2^2-ux_3^2-aux_4^2-2avx_3x_4,\\
  g_1(\xx) &= 2x_1x_2-2ux_3x_4-vx_3^2-avx_4^2
\end{align*}
in $\xx=(x_1, \dots, x_4)$ and some quadratic forms $\lambda_0(\yy),
\lambda_1(\yy)$ in $\yy=(y_1, \dots, y_4)$.  Then
\begin{multline*}
  \rho \cdot N_{K/L}(\xx)\cdot \sigma(N_{K/L}(\yy))\\ =
  g_0(\xx)\lambda_0(\yy)+ag_1(\xx)\lambda_1(\yy)+(g_0(\xx)\lambda_1(\yy)+g_1(\xx)\lambda_0(\yy))\sqrt a.
\end{multline*}
This gives an open embedding of $\T_U$ into the affine variety $\widetilde Y
\subset \AAA_k^9$ with coordinates $(t,\xx,\yy)$
defined by
\begin{equation*}
  t=g_0(\xx)\lambda_0(\yy)+ag_1(\xx)\lambda_1(\yy),\quad
  -1=g_0(\xx)\lambda_1(\yy)+g_1(\xx)\lambda_0(\yy).
\end{equation*}
Clearly $\widetilde Y \cong Y$ for $Y \subset \AAA_k^8$ with coordinates
$(\xx,\yy)$ defined by
\begin{equation}\label{eq:universal_torsors}
  -1=g_0(\xx)\lambda_1(\yy)+g_1(\xx)\lambda_0(\yy).
\end{equation}
We have a morphism $g: Y \to X$ defined by
\begin{equation*}
  (\xx,\yy) \mapsto (g_0(\xx)\lambda_0(\yy)+ag_1(\xx)\lambda_1(\yy),\xi\xx\yy).
\end{equation*}
We observe that $g : Y \to X$ and $f: \T \to X$ have the same
restriction to $\T_U$.

Note that $X$ is smooth over $k$. Also $Y$ is smooth over $k$ since
\begin{equation*}
  \Ybar \cong \{2\sqrt a =
  \sigma(\rho) w_1w_2w_3w_4-\rho w_1'w_2'w_3'w_4'\} \subset \AAA^8_\kbar,
\end{equation*}
using (\ref{eq:original_universal_torsors}). By
\cite[Proposition~2]{arXiv:1202.3567}, $\Pic(\Xbar)$ is torsion-free
of finite rank, hence $\T_U$ is a torsor under a torus over the
geometrically integral variety $U$, and hence $\T_U$ is geometrically
integral.

Therefore, we can apply Lemma~\ref{lem:passage} to get strong
approximation with Brauer--Manin obstruction for $X$ off $v_0$ once we
have shown that $Y$ satisfies strong approximation off $v_0$, which is
done in Proposition~\ref{prop:strong_approximation_torsor} below.
\end{proof}

The following result completes the proof of Theorem~\ref{thm:main}.

\begin{prop}\label{prop:strong_approximation_torsor}
  The affine variety $Y \subset \AAA^8_k$ defined by
  (\ref{eq:universal_torsors}) satisfies strong approximation off $v_0$.
\end{prop}

\begin{proof}
  Consider the projection $\pi: Y \to \AAA_k^4$ to the
  $\yy$-coordinate.  In \eqref{eq:universal_torsors}, we can write
  \begin{equation*}
    g_0(\xx)\lambda_1(\yy)+g_1(\xx)\lambda_0(\yy) = q_0(x_1,x_2)+q_1(x_3,x_4)
  \end{equation*}
  where $q_0,q_1$ are the following binary quadratic forms with coefficients
  in $k[y_1, \dots, y_4]$:
  \begin{align*}
    q_0(x_1,x_2)={}&\lambda_1(\yy)x_1^2+2\lambda_0(\yy)x_1x_2+a\lambda_1(\yy)x_2^2,\\
    q_1(x_3,x_4)={}&-(u\lambda_1(\yy)+v\lambda_0(\yy))x_3^2-2(av\lambda_1(\yy)+u\lambda_0(\yy))x_3x_4\\
    &-a(u\lambda_1(\yy)+v\lambda_0(\yy))x_4^2.
  \end{align*}
  Its discriminants are
  \begin{align*}
    \disc(q_0)&=a\lambda_1(\yy)^2-\lambda_0(\yy)^2=-N_{L/k}(\rho)N_{K/k}(\yy),\\
    \disc(q_1)&=(u^2-av^2)(a\lambda_1(\yy)^2-\lambda_0(\yy)^2) \\&=
    -N_{L/k}(u+v\sqrt a)N_{L/k}(\rho)N_{K/k}(\yy).
  \end{align*}
  For $N_{K/k}(\yy) \ne 0$, the binary forms $q_0,q_1$ have full rank, hence each fiber
  $Y_\yy$ is a three-dimensional quadric.

  Denote
  \begin{equation*}
    Z := \{\lambda_1(\yy)=\disc(q_0) = 0\} \cup
    \{u\lambda_1(\yy)+v\lambda_0(\yy)= \disc(q_1)=0\},
  \end{equation*}
  which is a closed subset of $\AAA_k^4$. For strong approximation, by
  smoothness of $Y$, it is enough to show that its open subset
  $\pi^{-1}(\AAA_k^4\setminus Z)$ satisfies strong approximation off
  $v_0$. Applying the natural fibration $\pi^{-1}(\AAA_k^4\setminus Z)$
  over $\AAA_k^4\setminus Z$, we only need to check the conditions for
  an application of the version of the fibration method stated in
  \cite[Proposition~3.1]{MR3007293}.

  It is clear that all geometric fibers of $\pi^{-1}(\AAA_k^4\setminus Z)$
  over $\AAA_k^4\setminus Z$ are nonempty and
  integral. Recall that $\disc(q_0)$ and $\disc(q_1)$ differ only by a
  constant in $k^\times$. Since $\disc(q_0)$ is irreducible of degree
  $4$ and $\lambda_1(\yy)$ has degree $2$, and similarly for the
  second part, we know that $Z$ has dimension $2$. Therefore
  $\AAA_k^4\setminus Z$ satisfies strong approximation off $v_0$ by
  \cite[Lemma 1.1]{arXiv:1403.1035}, hence the condition (i) holds.

  Let $W:=\{N_{K/k}(\yy)\neq 0\}\subset \AAA_k^4\setminus Z$. In the
  following, we will check the condition (ii) and (iii).

  For $v_0$ as above, we claim that $Y_\yy(k_{v_0})$ is not compact
  for any $\yy \in W(k_{v_0})$, hence the condition (iii) holds. For
  any $k$-point $r$ in $W$, the fiber $Y_r$ is a
  quadric of dimension $3$. Since $Y_r(k_{v_0})$ is not compact,
  $Y_r$ satisfies strong approximation off $v_0$ (see for example
  \cite[Theorem~3.7, \S 5.3]{MR2501421}), hence the condition (ii)
  holds.

  In the following we prove the claim. If $v_0$ is complex, this claim is
  obvious, so we assume that $v_0$ is real and consider everything in
  the following with respect to the corresponding real embedding of
  $k$. If one of $\disc(q_0)$ and $\disc(q_1)$ is negative, then one
  of $q_0$ and $q_1$ is indefinite, and the claim is true. So we can
  assume $\disc(q_0),\disc(q_1)>0$, and we claim that one of $q_0,q_1$
  is positive definite and the other is negative definite, i.e.,
  $\lambda_1(\yy)(-(u\lambda_1(\yy)+v\lambda_0(\yy)))<0$.

  For this, we show first that $u>0$. By assumption,
  $a\lambda_1(\yy)^2-\lambda_0(\yy)^2>0$ (so $a>0$) and $u^2-av^2>0$ (so
  $u-v\sqrt a$ and $u+v\sqrt a$ have the same sign). If $u<0$, then
  both $u-v\sqrt a$ and $u+v\sqrt a$ are negative, so all places of
  $K$ above $v_0$ are complex, and $N_{K/k}$ is positive definite. By
  unboundedness of $p(X(k_{v_0}))$, we have $c>0$. Therefore,
  $N_{L/k}(\rho)=c^{-1}N_{K/k}(\xi)>0$, hence $\disc(q_0)<0$,
  contradicting our assumption of definiteness of $q_0$. Hence $u>0$.

  Therefore, $u^2-av^2>0$ implies $u > |v|\sqrt a$. Furthermore,
  $a\lambda_1(\yy)^2-\lambda_0(\yy)^2>0$ implies $|\lambda_1(\yy)| >
  |\lambda_0(\yy)|/\sqrt a$. Therefore,
  \begin{align*}
    \lambda_1(\yy)(u\lambda_1(\yy)+v\lambda_0(\yy))&>
    (\sqrt a|v|)\frac{|\lambda_0(\yy)|}{\sqrt a}|\lambda_1(\yy)| + v\lambda_0(\yy)\lambda_1(\yy)\\
    &= |v\lambda_0(\yy)\lambda_1(\yy)|+v\lambda_0(\yy)\lambda_1(\yy)\ge 0,
  \end{align*}
  so one of $q_0,q_1$ is positive definite and the other is negative
  definite, and $Y_\yy(k_{v_0})$ is not compact also in this case.
\end{proof}

Our next goal is to compute the Brauer group of $X$ as in
Theorem~\ref{thm:main}. We start with some lemmas that will also be
used in our counterexample in Section~\ref{sec:counterexample}.

Let $X \subset \AAA^{n+1}_k$ be defined by (\ref{eq:norm}) for $s=1$, a finite
extension $K/k$ of fields of characteristic $0$ and a non-constant
polynomial $P(t) \in k[t]$. Recall that $X^\sm \subset X$ is the smooth locus
of $X$. If $P(t)$ is separable, then $X^\sm = X$.

The following lemma shows that $\Br(\Xbar^\sm)=0$.

\begin{lemma}\label{lem:transcendental_brauer}
  Let $k$ be an algebraically closed field of characteristic
  $0$. Let $P(t) = c(t-a_1)^{e_1} \cdots (t-a_r)^{e_r} \in k[t]$ with $r\ge 1$,
  $\gcd(e_1, \dots, e_r) = 1$ and $a_i \ne a_j$ for $i \ne j$. Let $X
  \subset \AAA^{n+1}$ be the affine variety defined by $P(t)=z_1\cdots
  z_n$. Then $\Br(X^\sm)=0$.
\end{lemma}

\begin{proof}
  We prove the claim by induction on $n$. If $n=1$, then $X^\sm = X
  \cong \AAA^1_k$. By Tsen's theorem, $\Br(X^\sm)=0$.

  For $n > 1$, consider the projection
  \begin{equation*}
    X \to \AAA^1_k,\quad (t,z_1, \dots, z_n) \mapsto z_n.
  \end{equation*}
  Let $\eta$ be the generic point of $\AAA^1_k$. Then the generic fiber
  $X^\sm_\eta$ is just the smooth locus of the affine variety over
  $k(z_n)$ defined by
  \begin{equation*}
    \frac{1}{z_n}P(t) = z_1\cdots z_{n-1}.
  \end{equation*}
  We have $\Br(X^\sm) \subset \Br(X^\sm_\eta)$ for any smooth variety,
  and we will show that the latter is trivial.

  Let $X^\sm_{\overline\eta} = X^\sm_\eta \times_{k(z_n)}
  \overline{k(z_n)}$.  Since $k[X^\sm_{\overline\eta}]^\times =
  \overline{k(z_n)}^\times$ by the same residue computation as in the
  proof of \cite[Proposition~2]{arXiv:1202.3567}, we have $H^2(k(z_n),
  k[X_{\overline\eta}]^\times) = \Br(k(z_n))$. Therefore, the Hochschild--Serre
  spectral sequence gives the exact sequence
  \begin{equation*}
    \Br(k(z_n)) \to \ker\left(\Br(X^\sm_\eta) \to \Br(X^\sm_{\overline\eta})\right)
    \to H^1(k(z_n), \Pic(X^\sm_{\overline\eta})).
  \end{equation*}
  By Tsen's theorem, $\Br(k(z_n))
  = 0$ since $k=\kbar$. By the induction hypothesis,
  $\Br(X^\sm_{\overline\eta}) = 0$.

  Since $\gcd(e_1, \dots, e_r)=1$, the Picard group
  $\Pic(X^\sm_{\overline\eta})$ is a constant torsion-free module
  (i.e., the natural action of $\Gamma_{k(z_n)}$ on it is trivial),
  hence we have $H^1(k(z_n), \Pic(X^\sm_{\overline\eta}))=0$.
  Therefore, the exact sequence above gives $\Br(X^\sm_\eta)=0$.
\end{proof}

For $X$ as above, let $U \subset X^\sm$ be the open subset defined by
$P(t) \ne 0$. Denote $\That = \kbar[U]^\times / \kbar^\times$ and
$\Mhat = \Div_{\Xbar\setminus \Ubar}(\Xbar)$. We have the natural exact sequence
\begin{equation}\label{eq:seq_gal_modules}
  0 \to \That \to \Mhat \to \Pic(\Xbar^\sm) \to 0.
\end{equation}

\begin{lemma}\label{lem:Br_formula}
  Let $k$ be a field of characteristic $0$ with
  $H^3(k,\kbar^\times)=0$. Assume that $P(t)=c g_1(t)^{e_1}\cdots
  g_r(t)^{e_r}$ for $c \in k^\times$, pairwise distinct irreducible
  monic polynomials $g_1(t), \dots, g_r(t) \in k[t]$ and $e_1, \dots,
  e_r$ positive integers with $\gcd(e_1, \dots, e_r)=1$. Let $X$ be
  defined by (\ref{eq:norm}) with $s=1$. Then
  \begin{equation*}
    \Br(X^\sm)/\Br_0(X^\sm) \cong H^1(k,\Pic(\Xbar^\sm)) \cong
    \ker\left(H^2(k,\That) \to H^2(k,\Mhat)\right).
  \end{equation*}
\end{lemma}

\begin{proof}
  By Lemma~\ref{lem:transcendental_brauer}, we have
  $\Br(\Xbar^\sm)=0$, hence $\Br(X^\sm) = \Br_1(X^\sm)$. Note that
  $\kbar[X^\sm]^\times = \kbar^\times$.

  By the Hochschild--Serre spectral sequence, we have the exact sequence
  \begin{equation*}
    \Br(k) \to \Br_1(X^\sm) \to H^1(k,\Pic(\Xbar^\sm)) \to H^3(k,\kbar^\times)=0,
  \end{equation*}
  hence $\Br(X^\sm)/\Br_0(X^\sm) \cong H^1(k,\Pic(\Xbar^\sm))$.

  Since $\Mhat$ is a permutation $\Gamma_k$-module,
  $H^1(k,\Mhat)=0$. Hence (\ref{eq:seq_gal_modules}) implies that
  \begin{equation*}
    0 \to H^1(k, \Pic(\Xbar^\sm)) \to H^2(k, \That) \to H^2(k,\Mhat)
  \end{equation*}
  is exact, and the result follows.
\end{proof}

To compute the Brauer group in the situation of
Theorem~\ref{thm:main}, we argue similarly as in
Lemma~\ref{lem:Br_formula}. The condition $H^3(k,\kbar^\times)=0$
holds when $k$ is a number field, for example.

\begin{prop}\label{prop:brauer}
  Let $k$ be a field of characteristic $0$ with
  $H^3(k,\kbar^\times)=0$. Let $K/k$ be of degree $4$ and $P(t) =
  c(t^2-a)$ irreducible over $k$ and split over $K$. Let $X$ be
  defined by (\ref{eq:norm}) with $s=1$.

  We have
  \begin{equation*}
    \Br(X)/\Br_0(X) =
    \begin{cases}
      0, &\text{if $K/k$ cyclic or non-Galois,}\\
      \ZZ/2\ZZ, &\text{otherwise.}
    \end{cases}
  \end{equation*}
\end{prop}

\begin{proof}
  First, suppose that $K/k$ is Galois. Since $\Pic(\Xbar)$ is split by $K$ and torsion-free, we have
  \begin{equation*}
    H^1(k, \Pic(\Xbar)) \cong H^1(K/k, \Pic(X_K))
  \end{equation*}
  by the inflation-restriction sequence.

  Let $L := k(\sqrt a)$. Note that $\Div_{X_K \setminus U_K} (X_K)
  \cong \ZZ[L/k] \otimes \ZZ[K/k]$, hence $H^i(K/k, \Div_{X_K
    \setminus U_K} (X_K)) = 0$ for $i>0$.  With the exact sequence
  \begin{equation}\label{eq:seq_That_Mhat_Pic_K}
    0 \to K[U]^\times/K^\times \to \Div_{X_K \setminus U_K} (X_K) \to \Pic(X_K) \to 0,
  \end{equation}
  this gives $H^1(K/k, \Pic(X_K)) \cong H^2(K/k,
  K[U]^\times/K^\times)$.

  To compute the latter, we consider the exact sequence
  \begin{equation}\label{eq:That_K_over_k}
    0 \to \ZZ \to \ZZ[L/k] \oplus \ZZ[K/k] \to K[U]^\times/K^\times \to 0
  \end{equation}
  from \cite[Proposition~2]{arXiv:1202.3567}. We consider
  \begin{multline*}
    H^2(K/k,\ZZ) \to H^2(K/k, \ZZ[L/k] \oplus \ZZ[K/k]) \to H^2(K/k, K[U]^\times/K^\times)\\ \to  H^3(K/k,\ZZ) \to H^3(K/k, \ZZ[L/k] \oplus \ZZ[K/k]).
  \end{multline*}

  We have $H^2(K/k,\ZZ[K/k])=0$ since this is an induced module, and
  $H^2(K/k,\ZZ[L/k]) \cong H^2(K/L,\ZZ)$ by Shapiro's lemma. Hence the
  first map is the natural map $H^2(K/k,\ZZ) \to H^2(K/L,\ZZ)$, which
  is surjective since any quartic Galois extension $K/k$ is abelian.

  Similarly, $H^3(K/k,\ZZ[K/k])=0$ and $H^3(K/k, \ZZ[L/k])
  \cong H^3(K/L,\ZZ)$. Since $\Gal(K/L)$ is cyclic,
  \begin{equation*}
    H^3(K/L,\ZZ) \cong H^1(K/L,\ZZ) \cong \Hom(\Gal(K/L),\ZZ) = 0.
  \end{equation*}
  Therefore, $H^3(K/k, \ZZ[L/k] \oplus \ZZ[K/k]) = 0$.

  Therefore, the third map in our part of the long exact sequence
  above is an isomorphism. Now
  \begin{equation*}
    H^3(K/k,\ZZ) =
    \begin{cases}
      0, &\text{if $K/k$ is cyclic,}\\
      \ZZ/2\ZZ, &\text{if $\Gal(K/k) \cong \ZZ/2\ZZ \times \ZZ/2\ZZ$.}
    \end{cases}
  \end{equation*}
  Indeed, in the first case, we argue as in the previous paragraph. In
  the second case, we can refer to a classical computation of Schur
  (see \cite[Corollary~2.2.12]{MR1200015}).

  Next, suppose that $K/k$ is non-Galois. By the inflation-restriction sequence, we have
  \begin{equation}\label{eq:inf_res}
    0 \to H^1(L/k, \Pic(\Xbar)^{\Gamma_L}) \to H^1(k,\Pic(\Xbar)) \to H^1(L,\Pic(\Xbar)).
  \end{equation}
  We will show that the first and third groups are trivial.

  For the triviality of $H^1(L/k, \Pic(\Xbar)^{\Gamma_L})$, we
  note that (\ref{eq:seq_gal_modules}) gives
  \begin{equation}\label{eq:seq_L}
    0 \to \That^{\Gamma_L} \to \Mhat^{\Gamma_L} \to \Pic(\Xbar)^{\Gamma_L} \to H^1(\Gamma_L, \That).
  \end{equation}
  Analogously to (\ref{eq:That_K_over_k}), we have over $\kbar$ the exact sequence
  \begin{equation}\label{eq:That_kbar_over_k}
    0 \to \ZZ \to \ZZ[L/k] \oplus \ZZ[K/k] \to \That \to 0.
  \end{equation}
  Since the action of $\Gamma_L$ on $\ZZ[L/k]$ is trivial, this
  exact sequence has a $\Gamma_L$-equivariant splitting, hence
  $\That$ is a permutation $\Gamma_L$-module, so
  $H^1(L, \That) = 0$.

  Therefore, (\ref{eq:seq_L}) is a short exact sequence, giving
  \begin{multline*}
    H^1(L/k,\Mhat^{\Gamma_L}) \to H^1(L/k, \Pic(\Xbar)^{\Gamma_L})\\ \to H^2(L/k, \That^{\Gamma_L}) \to H^2(L/k, \Mhat^{\Gamma_L}).
  \end{multline*}
  We note that $\Mhat \cong \ZZ[L/k] \otimes \ZZ[K/k]$, hence
  $\Mhat^{\Gamma_L} \cong \ZZ[L/k] \otimes
  \ZZ[K/k]^{\Gamma_L}$ is an induced
  $\Gal(L/k)$-module. Therefore, $H^i(L/k,\Mhat^{\Gamma_L})=0$
  for $i>0$, hence the second map in our part of the long exact
  sequence is an isomorphism.

  We have $\ZZ[K/k]^{\Gamma_L} \cong \ZZ[L/k]$. We choose $\sigma \in
  \Gamma_k$ and $\beta \in \Gamma_L$ such that $\Gamma_k/\Gamma_L =
  \{\Gamma_L, \sigma\Gamma_L\}$ and $\Gamma_L/\Gamma_K = \{\Gamma_K,
  \beta\Gamma_K\}$. Then
  \begin{equation*}
    \Gamma_k/\Gamma_K =
    \{\Gamma_K,\sigma\Gamma_K, \beta\Gamma_K,
    \sigma\beta\Gamma_K\}.
  \end{equation*}
  Now the orbits of the natural action of $\Gamma_L = \Gamma_K \cup
  \beta\Gamma_K$ on this are $\{\Gamma_K,\beta\Gamma_K\}$ (since
  $\Gamma_K$ is normal in $\Gamma_L$ because of $[K:L]=2$) and
  $\{\sigma\Gamma_K, \sigma\beta\Gamma_K\}$ (since the action of
  $\Gamma_L$ on this set is non-trivial because
  $\sigma\beta\sigma^{-1}$ maps $\sigma\Gamma_K$ to
  $\sigma\beta\Gamma_K$ and lies in $\Gamma_L$ because $\Gamma_L$ is
  normal in $\Gamma_k$ since $[L:k]=2$). Therefore,
  $\ZZ[K/k]^{\Gamma_L} = \langle \Gamma_K+\beta\Gamma_K,
  \sigma\Gamma_K+\sigma\beta\Gamma_K\rangle$ is of rank $2$, with a
  non-trivial action of $\Gal(L/k)$.

  Taking $\Gamma_L$-invariants of (\ref{eq:That_kbar_over_k}) gives
  \begin{equation*}
    0 \to \ZZ \to \ZZ[L/k] \oplus \ZZ[L/k] \to \That^{\Gamma_L} \to H^1(L,\ZZ)=0.
  \end{equation*}
  The long exact sequence gives
  \begin{equation*}
    0=H^2(L/k, \ZZ[L/k] \oplus \ZZ[L/k]) \to H^2(L/k, \That^{\Gamma_L}) \to H^3(L/k,\ZZ)=0,
  \end{equation*}
  where the latter is trivial because $L/k$ is cyclic. Hence $H^2(L/k,
  \That^{\Gamma_L})=0$, which implies $H^1(L/k,
  \Pic(\Xbar)^{\Gamma_L})=0$ in (\ref{eq:inf_res}).

  For the triviality of $H^1(L,\Pic(\Xbar))$ in (\ref{eq:inf_res}), we
  write $K = k(\sqrt{u+v\sqrt a})$ for some $u \in k$ and $v \in
  k^\times$. Then $K\otimes_k L \cong K \oplus K'$ with $K' =
  k(\sqrt{u-v\sqrt a})$.  We note that $X \times_k L$ is defined by
  the equation
  \begin{equation*}
    c(t-\sqrt a)(t+\sqrt a) = N_{K/L}(\zz_1)N_{K'/L}(\zz_2).
  \end{equation*}
  Over $\kbar$, we have $N_{K/L}(\zz_1) = u_1u_2$ and $N_{K'/L}(\zz_2) = u_3u_4$.

  As a sequence of $\Gamma_L$-modules, (\ref{eq:That_kbar_over_k})
  becomes
  \begin{equation*}
    0 \to \ZZ \to \ZZ^2 \oplus \ZZ[K/L] \oplus \ZZ[K'/L] \to \That \to 0,
  \end{equation*}
  hence $\That \cong \ZZ \oplus \ZZ[K/L] \oplus \ZZ[K'/L]$, with basis
  $(t-\sqrt a,u_1,u_2,u_3,u_4)$. Therefore, $H^2(L,\That) \cong
  H^2(L,\ZZ) \oplus H^2(K,\ZZ) \oplus H^2(K',\ZZ)$ by Shapiro's lemma.

  Then
  \begin{equation*}
    D_i^+ = \{t=\sqrt a, u_i=0\},\quad D_i^- = \{t=-\sqrt a, u_i=0\}
  \end{equation*}
  for $i=1, \dots, 4$ are a basis of $\Mhat$. Considering the action of
  $\Gamma_L$ on this basis shows that
  \begin{align*}
    &\ZZ D_1^+ \oplus \ZZ D_2^+ \cong \ZZ D_1^- \oplus \ZZ D_2^- \cong \ZZ[K/L],\\
    &\ZZ D_3^+ \oplus \ZZ D_4^+ \cong \ZZ D_3^- \oplus \ZZ D_4^- \cong \ZZ[K'/L],
  \end{align*}
  hence $\Mhat \cong (\ZZ[K/L] \oplus \ZZ[K'/L])^2$. Therefore,
  $H^2(L,\Mhat) \cong (H^2(K,\ZZ) \oplus H^2(K',\ZZ))^2$, again by Shapiro's lemma.

  The map
  \begin{equation*}
    \That \to \Mhat, \quad t-\sqrt a\mapsto D_1^++D_2^++D_3^++D_4^+,\ u_i \mapsto D_i^++D_i^-
  \end{equation*}
  from (\ref{eq:seq_gal_modules}) induces a map $H^2(L,\That) \to
  H^2(L, \Mhat)$ that can be explicitly described as
  \begin{equation*}
    (\chi_1,\chi_2,\chi_3) \mapsto (\chi_{1|K} + \chi_2, \chi_{1|K'}+\chi_3,\chi_2,\chi_3).
  \end{equation*}

  The sequence (\ref{eq:seq_gal_modules}) gives
  \begin{equation*}
    0=H^1(L,\Mhat) \to H^1(L,\Pic(\Xbar)) \to H^2(L,\That) \xrightarrow{\phi} H^2(L,\Mhat)
  \end{equation*}
  since $\Mhat$ is a permutation $\Gamma_L$-module. To deduce
  triviality of $H^1(L,\Pic(\Xbar))$, we show that $\phi$ is
  injective. If $(\chi_1,\chi_2,\chi_3) \in \ker\phi$, then $\chi_2$,
  $\chi_3$ and hence also $\chi_{1|K}$, $\chi_{1|K'}$ are
  trivial. Since $K \cap K'=L$, the latter implies $\chi_1=0$, proving
  injectivity of $\phi$.

  Using Lemma~\ref{lem:Br_formula} and (\ref{eq:inf_res}), we deduce
  \begin{equation*}
    \Br(X)/\Br_0(X)\cong H^1(k, \Pic(\Xbar))=0.\qedhere
  \end{equation*}
\end{proof}

\begin{cor}\label{cor:strong_approximation}
  In the situation of Theorem~\ref{thm:main}, assume additionally that
  the extension $K/k$ is cyclic or non-Galois. Then strong
  approximation holds for $X$ off $v_0$.
\end{cor}

\begin{proof}
  This follows directly from Theorem~\ref{thm:main} and Proposition~\ref{prop:brauer}.
\end{proof}

\begin{cor}\label{cor:hasse_principle}
  Let $K/k$ be an extension of number fields of degree $4$, and let
  $\omega_1, \dots, \omega_4 \in \OO_K$ be a basis of $K$ over $k$.
  Let $\XX \subset \AAA^5_{\OO_k}$ be the affine scheme defined by
  \begin{equation*}
    at^2-b = N_{K/k}(z_1\omega_1+\dots+z_4\omega_4)
  \end{equation*}
  for some $a,b \in \OO_k$ such that $at^2-b \in k[t]$ is irreducible
  and splits in $K$. For $X := \XX \times_{\OO_k} k$ and $p : X \to
  \AAA^1_k$ the projection to the $t$-coordinate, assume that $\prod_{v
    \in \infty_k} p(X(k_v))$ is not bounded.

  If
  \begin{equation*}
    \left(\prod_{v \in \infty_k} X(k_v) \times \prod_{v \notin
        \infty_k} \XX(\OO_v)\right)^{\Br(X)} \ne \emptyset,
  \end{equation*}
  then $\XX(\OO_k) \ne \emptyset$, i.e., the Hasse principle with
  Brauer--Manin obstruction holds for integral points on $\XX$.  In
  particular, if $K/k$ is cyclic or non-Galois, then the Hasse
  principle holds for integral points on $\XX$.
\end{cor}

\begin{proof}
  This follows from Theorem~\ref{thm:main} as in
  Remark~\ref{rem:smooth_sa_hp}. For $K/k$ cyclic or non-Galois, we
  have $\Br(X) = \Br_0(X)$ by Proposition~\ref{prop:brauer}.
\end{proof}

\section{Totally split polynomials represented by
  norms}\label{sec:totally_split}

For the proof of strong approximation for $X$ as in
Theorem~\ref{thm:totally_split}, we determine its universal torsors in
Proposition~\ref{prop:universal_torsors} below. If the factors of
$P(\ttt)$ are linear forms, these turn out to be essentially the
varieties $\mathscr{V}$ that satisfy weak approximation by
\cite[Theorem~1.3]{arXiv:1307.7641} based on
\cite[Theorem~5.2]{arXiv:1307.7641}; for
Theorem~\ref{thm:totally_split}, we use a generalization,
Theorem~\ref{thm:asymptotic} below.  By our Descent
Lemma~\ref{lem:passage}, it then remains to prove strong approximation
for $\mathscr{V}$, which we do in
Proposition~\ref{prop:strong_approximation}.

\begin{prop}\label{prop:universal_torsors}
  Let $k$ be a field of characteristic $0$. Let $K/k$ be a field
  extension of degree $n$. Let $P(\ttt)=c g_1(\ttt)^{e_1}\cdots
  g_r(\ttt)^{e_r}$ with $c \in k^\times$ and pairwise non-proportional
  linear polynomials $g_i(\ttt) \in k[t_1, \dots, t_s]$ in $s \ge 2$
  variables, and $e_1, \dots, e_r \in \ZZ_{>0}$. Let $X$ be defined by
  (\ref{eq:norm}).

  Universal torsors over $X^\sm$ exist if and only if there are
  $\lambda_1, \dots, \lambda_r \in k^\times$ and $\xi \in K^\times$
  satisfying
  \begin{equation}\label{eq:torsor_existence}
    c\lambda_1^{e_1}\cdots \lambda_r^{e_r} = N_{K/k}(\xi).
  \end{equation}

  Then restrictions $\T_U \subset \AAA^{rn+s}_k$ of universal torsors
  $f: \T \to X^\sm$ to $U:=X \cap \{P(\ttt) \ne 0\}$ are geometrically
  rational and defined by
  \begin{equation*}
    N_{K/k}(\zz_i) = \lambda_i^{-1}g_i(\ttt) \ne 0,
  \end{equation*}
  for $i=1, \dots, r$, for some $\lambda_1, \dots, \lambda_r\in
  k^\times$ and $\xi \in K^\times$ satisfying
  (\ref{eq:torsor_existence}). The map $f : \T \to X^\sm$ is defined
  on $\T_U$ by $(\ttt, \zz_1, \dots, \zz_r) \mapsto (\ttt,
  \xi\zz_1^{e_1}\cdots\zz_r^{e_r})$.
\end{prop}

\begin{proof}
  The proof proceeds along the lines of \cite[Theorem~2.2]{MR1961196}
  and \cite[Proposition~2]{arXiv:1202.3567}, based on the
  local description of universal torsors in \cite[Theorem~2.3.1,
  Corollary~2.3.4]{MR89f:11082}. Since $\T_U$ is defined over $\kbar$ by
  $u_{i,1}\cdots u_{i,n} = \lambda_i^{-1}g_i(\ttt) \ne 0$
  for $i=1, \dots, r$, it is geometrically rational.
\end{proof}

To prove strong approximation on the universal torsors as in
Proposition~\ref{prop:universal_torsors} in the situation of
Theorem~\ref{thm:totally_split}, we use a generalization of the main
analytic result \cite[Theorem~5.2]{arXiv:1307.7641} of Browning and
Matthiesen from linear forms to (not necessarily homogeneous) linear
polynomials; see also \cite[Remark~5.3]{arXiv:1307.7641}. To state the
result, we introduce some notation.

Let $K$ be a number field of degree $n$, and let $N_{K/\QQ}(\zz) \in
\ZZ[z_1, \dots, z_n]$ be an associated norm form with integral
coefficients. Let $f_1(\ttt), \dots, f_r(\ttt) \in \ZZ[t_1, \dots,
t_s]$ be linear polynomials that are pairwise \emph{affinely
  independent} over $\QQ$, i.e., for $i \ne j \in \{1, \dots, r\}$,
the homogeneous parts $f_i(\ttt)-f_i(0)$ and $f_j(\ttt)-f_j(0)$ are
not proportional over $\QQ$.
Consider the system of equations
\begin{equation}\label{eq:integral_torsor}
  N_{K/\QQ}(\zz_i) = f_i(\ttt),\quad (i=1, \dots, r).
\end{equation}

Let $\frak{D}_+ \subset \RR^n$ be the fundamental domain for the action of the
free part of the norm-$1$-subgroup of the unit group of $K$ as in
\cite[equation~(2.3)]{arXiv:1307.7641}.
Let $\frak{K} \subset [-1,1]^s \subset \RR^s$ be a convex bounded set with
$|f_i(T\frak{K})| \le T$, for $1 \le i \le r$ and sufficiently large
$T \in \RR$.

From \cite[Definition~5.1]{arXiv:1307.7641}, recall the definition of
the representation function
\begin{equation*}
  R(m;\frak{D}_+,\zz',M):= \mathbf{1}_{m \ne 0} \cdot \#\left\{\zz'' \in \ZZ^n \cap \frak{D}_+ :
    \begin{aligned}
      &N_{K/\QQ}(\zz'') = m,\\
      &\zz'' \equiv \zz' \pmod M
    \end{aligned}
  \right\},
\end{equation*}
for $m \in \ZZ$, $M \in \ZZ_{>0}$ and $\zz' \in \ZZ^n$.

Let $q'=(\ttt',\zz_1', \dots, \zz_r') \in \ZZ^{rn+s}$ and $M \in \ZZ_{>0}$.
Then
\begin{equation}\label{eq:counting_function}
N(T):=\sum_{\substack{\ttt'' \in \ZZ^s \cap T\frak{K}\\\ttt'' \equiv \ttt' \pmod M}}
  \prod_{i=1}^r R(f_i(\ttt'');\frak{D}_+,\zz_i',M)
\end{equation}
is the number of solutions $q''=(\ttt'',\zz_1'', \dots, \zz_r'') \in
\ZZ^{rn+s}$ of (\ref{eq:integral_torsor}) lying in $T\frak{K} \times
(\frak{D}_+)^r$ with $q'' \equiv q' \pmod M$ and $f_1(\ttt'')\cdots
f_r(\ttt'') \ne 0$.

Let $r_1$ resp.\ $r_2$ be the number of real resp. complex places of
$K$, let $D_K$ be its discriminant, and let $R_K^{(+)}$ be its
modified regulator as in \cite[Remark~5.4]{arXiv:1307.7641}. Let
\begin{equation*}
c_K:=\frac{2^{r_1-1}(2\pi)^{r_2}R_K^{(+)}}{\sqrt{|D_K|}}.
\end{equation*}

\begin{theorem}\label{thm:asymptotic}
  Let $f_1, \dots, f_r \in \ZZ[t_1, \dots, t_s]$ be linear polynomials
  that are pairwise affinely independent, let $K$ be a number field,
  let $M \in \ZZ$, and let $q'=(\ttt',\zz_1', \dots, \zz_r') \in
  \ZZ^{rn+s}$. Let $\epsilon$ be an arbirary element of $\{\pm\}^r$ if
  $K$ is not totally imaginary, and let $\epsilon = (+, \dots, +)$ if
  $K$ is totally imaginary. Let $\frak{K} \subset [-1,1]^s$ be a
  convex bounded set whose closure is contained in
  \begin{equation*}
    \frak{Q}_\epsilon:=\{\ttt \in \RR^s : \text{$0 < \epsilon_i(f_i(\ttt)-f_i(0)) < 1$ for $i=1, \dots, r$}\}.
  \end{equation*}

  For $N(T)$ as in (\ref{eq:counting_function}), we have
  \begin{equation*}
    N(T) = \beta_\infty \prod_p \beta_p \cdot T^s  + o(T^s), \quad (T \to \infty),
  \end{equation*}
  where
  \begin{equation*}
    \beta_\infty := c_K^r\vol(\frak{K}),
  \end{equation*}
  and $\prod_p \beta_p$ is absolutely convergent, with
  \begin{equation*}
    \beta_p:=\lim_{m \to \infty} \frac{1}{p^{ms}}
    \sum_{\substack{\ttt''' \in (\ZZ/p^m\ZZ)^s\\\ttt''' \equiv \ttt' \pmod{p^{v_p(M)}}}}
    \prod_{i=1}^r \frac{\rho(p^m, f_i(\ttt'''), \zz_i'; p^{v_p(M)})}{p^{m(n-1)}}
  \end{equation*}
  and
  \begin{equation*}
    \rho(p^m,A,\zz';p^{v_p(M)}) := \#\left\{\zz''' \in (\ZZ/p^m\ZZ)^n :
      \begin{aligned}
        &N_{K/\QQ}(\zz''') \equiv A \pmod{p^m},\\
        &\zz''' \equiv \zz' \pmod{p^{v_p(M)}}
      \end{aligned}
    \right\}
  \end{equation*}
  for $m\ge v_p(M)$ and $A \in \ZZ/p^m\ZZ$.
\end{theorem}

\begin{proof}
  The proof follows \cite[\S 6--\S 10]{arXiv:1307.7641}.
  We note that our condition on $\frak{K}$ implies that
  \begin{equation}\label{eq:inequality}
    0 < \epsilon_if_i(T\frak{K}) < T
  \end{equation}
  for $i=1, \dots, r$ and sufficiently large $T$. This allows us to
  apply \cite[Proposition~7.10]{arXiv:1307.7641} as in
  \cite[Proposition~8.2]{arXiv:1307.7641}, but without decomposing
  $\frak{K}$ (in the inhomogeneous case, a decomposition that allows a
  direct application of \cite[Proposition~7.10]{arXiv:1307.7641} does
  not seem to exist).
\end{proof}

\begin{prop}\label{prop:strong_approximation}
  Let $s \ge 2$.  Let $Y \subset \AAA_{\QQ}^{rn+s}$ be the variety
  defined by
  \begin{equation}\label{eq:integral}
    N_{K/\QQ}(\zz_i) = \lambda_i^{-1}g_i(\ttt), \quad (i=1, \dots, r)
  \end{equation}
  with $g_i(t_1, \dots, t_s) \in \QQ[t_1, \dots, t_s]$ pairwise affinely
  independent linear polynomials and $\lambda_i \in \QQ^\times$.

  If $K$ is not totally imaginary or if the polytope
  \begin{equation*}
    Q:=\{\ttt \in \RR^s : \text{$\lambda_i^{-1} g_i(\ttt) > 0$ for $i=1, \dots, r$}\}
  \end{equation*}
  contains balls of arbitrarily large radius, then smooth strong
  approximation holds for $Y$ off $\infty$.
\end{prop}

\begin{proof}
  We may assume that the norm form $N_{K/\QQ} \in \ZZ[z_1, \dots,
  z_n]$ has integral coefficients. By rescaling $t_1, \dots, t_s$ and
  $\zz_1, \dots, \zz_r$, we may assume that
  \begin{equation*}
    \lambda_i^{-1}g_i(\ttt) \in \ZZ[t_1, \dots, t_s]
  \end{equation*}
  for $i=1, \dots, r$ are linear forms with integral coefficients.  We
  define the integral model $\YY \subset \AAA^{rn+s}_\ZZ$ of $Y$ by the
  same equations, but considered over $\ZZ$.

  For smooth strong approximation on $Y$, we must show: for any finite set of
  places $S \subset \Omega_\QQ \setminus \{\infty\}$, any $(q_v) =
  (\ttt_v,\zz_{1,v}, \dots, \zz_{r,v}) \in Y^\sm(\AAA_\QQ)$ with $q_v \in
  \YY(\ZZ_v)$ for all $v \notin S \cup \{\infty\}$, there is a $q \in
  Y^\sm(\QQ)$ with $q \in \YY(\ZZ_v)$ for all $v \notin S \cup
  \{\infty\}$ and $q$ arbitrarily close to $q_v$ for all $v \in S$.
  This is the same as finding $q \in Y^\sm(\QQ) \cap \YY(\ZZ_S)$ arbitrarily close
  to $q_v$ for all $v \in S$.

  By the implicit function theorem, we may assume that the given
  $(q_v)$ satisfies
  \begin{equation}\label{eq:non-zero}
    g_1(\ttt_v)\cdots g_r(\ttt_v) \ne 0
  \end{equation}
  for all $v \in \Omega_\QQ$.

  Let $C \in \ZZ$ with $C^{-1} \in \ZZ_S$ (i.e., all prime factors
  of $C$ lie in $S$) be such that
\begin{equation*}
q_v'=(\ttt_v',\zz_{1,v}', \dots, \zz_{r,v}'):=
(C\ttt_v,C\zz_{1,v}, \dots,C\zz_{r,v}) \in \ZZ_v^{rn+s}
\end{equation*}
for all $v \in S$.

Let $Y'$ be the variety
obtained from $Y$ by the change of coordinates replacing $(\ttt,\zz_1, \dots, \zz_r)$ by $(C\ttt, C\zz_1,
\dots, C\zz_r)$. An integral model $\YY'$ of $Y'$ is given by
\begin{equation}\label{eq:integral_new}
  N_{K/\QQ}(\zz_i)=f_i(\ttt), \quad (i=1, \dots, r)
\end{equation}
with $f_i(\ttt):=C^n\lambda_i^{-1}g_i(C^{-1}\ttt) \in \ZZ[\ttt]$. This
maps $(q_v) \in Y^\sm(\AAA_\QQ)$ to $(q_v') \in Y'^\sm(\AAA_\QQ)$ as above with
$q_v' \in \YY'(\ZZ_v)$ for all $v \in S$ (and of course still $q_v'
\in \YY'(\ZZ_v)$ for all $v \notin S \cup \{\infty\}$).

By strong approximation on $\AAA^{rn+s}_\ZZ$, we find $q'=(\ttt',\zz_1',
\dots, \zz_r') \in \ZZ^{rn+s}$ arbitrarily close to $q_v$ for all $v
\in S$.

Now we are looking for a point $q''=(\ttt'',\zz_1'',\dots,\zz_r'') \in
Y'^\sm(\QQ) \cap \YY'(\ZZ)$ very close to $q'$ in the $v$-adic topology
for all $v \in S$. This translates into the condition
\begin{equation}\label{eq:congruence}
  q'' \equiv q' \pmod M
\end{equation}
for a positive integer $M=\prod_{p \in S} p^{m'}$ for some
sufficiently large integer $m'$.  Once we have found such a $q'' \in
Y'^\sm(\QQ) \cap \YY'(\ZZ)$, then this is also very close to $q_v' \in
\YY'(\ZZ_v)$ for all $v \in S$, and then $q:=(C^{-1}\ttt'',
C^{-1}\zz_1'', \dots, C^{-1}\zz_r'') \in Y^\sm(\QQ) \cap \YY(\ZZ_S)$
is very close to $q_v$ for all $v \in S$, completing the proof.

By definition of $N(T)$ in (\ref{eq:counting_function}), where the
condition $f_1(\ttt'')\cdots f_r(\ttt'') \ne 0$ ensures $q'' \in
Y'^\sm(\QQ)$, it is enough to show $N(T) > 0$ for some $T$.

There is an $\epsilon \in \{\pm\}^r$ and a set $\frak{K}$ with
positive volume satisfying the conditions of
Theorem~\ref{thm:asymptotic}. Indeed, if $K$ is not totally imaginary,
we choose $\epsilon \in \{\pm\}^r$ such that
\begin{equation*}
  Q_\epsilon:=\{\ttt \in \RR^s : \text{$\epsilon_i(f_i(\ttt)-f_i(0)) > 0$ for $i=1, \dots, r$}\},
\end{equation*}
(which is a cone whose vertex is the origin) is non-empty; such an
$\epsilon$ clearly exists. If $K$ is totally imaginary, we must choose
$\epsilon=(+, \dots, +)$. We note that our condition on $Q$ is
equivalent to the condition that
\begin{equation*}
  Q':=\{\ttt \in \RR^s : \text{$f_i(\ttt) > 0$ for $i=1, \dots, r$}\}
\end{equation*}
contains balls of arbitrarily large radius.  For $R \in \RR_{>0}$, let
$\ttt_R$ be the center of a ball of radius $R$ contained in $Q'$. For
sufficiently large $R$, we have $f_i(\ttt_R) > f_i(0)$, hence $\ttt_R$
lies in the cone $Q_\epsilon$ for $\epsilon=(+, \dots, +)$, which is
therefore non-empty.  For any $K$, a point $\ttt$ in $Q_\epsilon$
sufficiently close to the origin satisfies $f_i(\ttt) < 1$ for $i=1,
\dots, r$, hence lies in $\frak{Q}_\epsilon$ as in
Theorem~\ref{thm:asymptotic}, and the same is true for a small enough
closed ball $\frak{K}$ around $\ttt$.

Hence an application of Theorem~\ref{thm:asymptotic} with this
$\frak{K}$ gives $N(T) = \beta_\infty \prod_p \beta_p \cdot T^s +
o(T^s)$ for large enough $T$, where $\beta_\infty =
c_K^r\vol(\frak{K}) > 0$.

Finally, we show that $\prod_p\beta_p > 0$. By
\cite[Proposition~5.5]{arXiv:1307.7641} (which also holds for our
inhomogeneous linear polynomials) there is an $L'>0$ such that $\prod_{p >
  L'} \beta_p > 0$. Now we show that $\beta_p > 0$ for any prime $v=p$. Let
\begin{equation*}
  m':=2(v_p(M)+\max_{i=1, \dots, r} v_p(f_i(\ttt_p')) + v_p(n))+1.
\end{equation*}
Let $q''' = (\ttt''',\zz_1''', \dots, \zz_r''') \in \ZZ^{rn+s}$ be
such that $q''' \equiv q_p' \pmod{p^{m'}}$.  Then
$\rho(p^{m'},f_i(\ttt'''), \zz_i''',p^{v_p(M)}) > 0$. Since $f_i(\ttt''') \ne 0 \in
\ZZ/p^{m'}\ZZ$ because $f_i(\ttt_p') \ne 0 \in \ZZ_p$ by (\ref{eq:non-zero}),
we can apply \cite[Lemma~3.4]{arXiv:1307.7641} to obtain
\begin{equation*}
  \frac{\rho(p^m,f_i(\ttt''''), \zz_i''';p^{v_p(M)})}{p^{m(n-1)}}=\frac{\rho(p^{m'},f_i(\ttt'''), \zz_i''';p^{v_p(M)})}{p^{m'(n-1)}} \ge \frac{1}{p^{m'(n-1)}}
\end{equation*}
for all $m\ge m'$ and all $\ttt'''' \in (\ZZ/p^m\ZZ)^s$ satisfying
$\ttt'''' \equiv \ttt''' \pmod{p^{m'}}$. Since $\ttt'$ was chosen very
close to $\ttt_p'$ for $p \in S$, and since $v_p(M)=0$ for $p \notin
S \cup \{\infty\}$, we note that $\ttt''' \equiv \ttt_p'
\pmod{p^{m'}}$ implies $\ttt''' \equiv \ttt'
\pmod{p^{v_p(M)}}$. Therefore, the $m$-th approximation $\beta_{p,m}$
of $\beta_p$ has at least $p^{(m-m')s}$ summands from these $\ttt''''$
that are at least as large as $p^{-ms}\cdot (p^{-m'(n-1)})^r$, hence $\beta_{p,m}
\ge p^{-m'(s+r(n-1))}$. Therefore, $\beta_p > 0$ for all primes $p$.
\end{proof}

The complement of the hyperplanes $\{g_1(\ttt)=0\}$, \dots, $\{g_r(\ttt)=0\}$ in affine $s$-space
consists of open polytopes. The following lemma shows that essentially every
other such polytope occurs as the interior of a connected component of
$p(X^\sm(\RR))$ if $K$ is totally imaginary.

\begin{lemma}\label{lem:connected_components}
  Let $X$ be as in Theorem~\ref{thm:totally_split}, with $P(\ttt) =
  cg_1(\ttt)^{e_1}\cdots g_r(\ttt)^{e_r}$ as in
  Proposition~\ref{prop:universal_torsors}, for $k=\QQ$.

  If $K$ is not totally imaginary, then $p(X^\sm(\RR)) = \RR^s$.
  If $K$ is totally imaginary, then the connected components of $p(X^\sm(\RR))$
  are precisely the sets
  \begin{equation*}
    Q_\epsilon := \left\{\ttt \in \RR^s :
    \begin{aligned}
      &\text{$g_i(\ttt) = 0$ for at most one $i\in \{1, \dots, r\}$ with
        $e_i=1$}\\
      &\text{$\epsilon_i g_i(\ttt) > 0$ for all other $i \in \{1, \dots, r\}$}
    \end{aligned}
    \right\}
  \end{equation*}
  for all $\epsilon=(\epsilon_1, \dots, \epsilon_r)\in \{\pm\}^r$ such that
  $(\epsilon_1)^{e_1}\cdots (\epsilon_r)^{e_r}\cdot c > 0$.
\end{lemma}

\begin{proof}
  Since $\Xbar$ is defined by
  $c g_1(\ttt)^{e_1} \cdots g_r(\ttt)^{e_r} = u_1\cdots u_n$,
  the singular locus of $\Xbar$ is the union of all $\{g_i(\ttt) =
  g_j(\ttt) = u_l = u_m = 0\}$ for $l \ne m \in \{1, \dots, n\}$ and
  $i,j \in \{1, \dots, r\}$ with either $i=j$ and $e_i>1$ or $i \ne
  j$.

  If $K$ is not totally imaginary, the norm form $N_{K/\QQ}$ is
  indefinite over $\RR$, and for any $\ttt \in \RR^s$, the number
  $P(\ttt)$ can be represented by $N_{K/\QQ}(\zz)$ for some $\zz \in
  \RR^n$ with $(\ttt,\zz) \in X^\sm(\RR)$.

  If $K$ is totally imaginary, the norm form $N_{K/\QQ}$ is positive definite
  over $\RR$, hence for any $(\ttt,\zz) \in X(\RR)$ with $P(\ttt)=0$, we have
  $\zz=0$, hence $\ttt \in p(X^\sm(\RR))$ if and only if $g_i(\ttt)=0$ for at
  most one $i \in \{1, \dots, r\}$ that satisfies additionally $e_i=1$. This
  clearly splits into the connected components described in the statement of
  the lemma.
\end{proof}

\begin{proof}[Proof of Theorem~\ref{thm:totally_split}]
  We want to apply Lemma~\ref{lem:passage}. Let $U := X \cap \{P(\ttt)
  \ne 0\}$. Let $f: \T \to X^\sm$ be a universal torsor, with local
  description $\T_U \to U$ as in
  Proposition~\ref{prop:universal_torsors} with associated $\lambda_1,
  \dots, \lambda_r \in \QQ^\times$ and $\xi \in K^\times$.  Also by
  Proposition~\ref{prop:universal_torsors}, $\T_U$ is geometrically
  rational, hence geometrically integral.

  Let $Y \subset \AAA_\QQ^{rn+s}$ be defined by
  \begin{equation*}
    N_{K/\QQ}(\zz_i) = \lambda_i^{-1}g_i(\ttt),\quad (i=1, \dots, r).
  \end{equation*}
  By Proposition~\ref{prop:universal_torsors}, we have an open
  immersion $\T_U \subset Y$.

  The morphism $g : Y \to X$ defined by $(\ttt, \zz_1, \dots, \zz_r)
  \mapsto (\ttt, \xi\zz_1^{e_1}\cdots\zz_r^{e_r})$ clearly agrees on
  $\T_U$ with $f : \T \to X^\sm$.

  If $K$ is not totally imaginary, every such $Y$ satisfies smooth strong
  approximation off $\infty$ by Proposition~\ref{prop:strong_approximation},
  hence Lemma~\ref{lem:passage} implies
  Theorem~\ref{thm:totally_split}. Furthermore, $C=p(X^\sm(\RR)) = \RR^s$ by
  Lemma~\ref{lem:connected_components}.

  If $K$ is totally imaginary, Proposition~\ref{prop:strong_approximation}
  gives smooth strong approximation off $\infty$ not necessarily for all
  $Y$. Following the proof of Lemma~\ref{lem:passage}, we see that $X^\sm(k)$
  is dense in the image of $V^{\Br(X^\sm)}$ in $X(\AAA_k^f)$, where $V$ is the
  set of all $(p_v) \in X^\sm(\AAA_k)$ that can be lifted to $(r_v) \in
  \T(\AAA_k)$ for a universal torsor $f : \T \to X^\sm$ such that the
  associated $Y$ satisfies strong approximation off $\infty$.

  To complete the proof of Theorem~\ref{thm:totally_split}, it remains
  to show that $p^{-1}(C) \times X^\sm(\AAA_k^f)$ is contained in
  $V$. Indeed, consider $(p_v)=(\ttt_v,\zz_v) \in X^\sm(\AAA_k)$ with
  $\ttt_\infty$ in an unbounded connected component $Q_\epsilon$ of
  $C$ for some $\epsilon \in \{\pm\}^r$ as in
  Lemma~\ref{lem:connected_components}. Then we have its lift $(r_v)
  \in \T(\AAA_k)$ for some universal torsor $f : \T \to X^\sm$, and
  very close to it $(r_v')=(\ttt_v',\zz_{1,v}', \dots, \zz_{r,v}') \in
  \T_U(\AAA_k)$. By choosing $(r_v')$ close enough to $(r_v)$, we know
  that $(f(r_v'))=(\ttt_v', \xi\zz_{1,v}'^{e_1}\cdots
  \zz_{r,v}'^{e_r}) \in U(\AAA_k)$ has $\ttt_\infty'$ on the same
  $Q_\epsilon$ as $\ttt_\infty$. Since $P(\ttt_\infty') \ne 0$, we
  have $\epsilon = (\epsilon_1,\dots, \epsilon_r)$ with $\epsilon_i =
  \sgn(g_i(\ttt'_\infty)) \in \{\pm\}$. Since $(r_v')$ satisfies the
  equations defining $\T_U$ as in
  Proposition~\ref{prop:universal_torsors}, we have
  $\lambda_i^{-1}g_i(\ttt_\infty') = N_{K/\QQ}(\zz_{i,\infty}') >
  0$, hence $\sgn(\lambda_i^{-1}) = \sgn(g_i(\ttt_\infty')) =
  \epsilon_i$. Therefore, the associated $Q$ in
  Proposition~\ref{prop:strong_approximation} is precisely the
  interior of $Q_\epsilon$. Since $Q_\epsilon$ contains balls of arbitrarily
  large radius by
  assumption, the same holds for $Q$, and
  Proposition~\ref{prop:strong_approximation} tells us that $Y$
  satisfies smooth strong approximation, hence $(p_v) \in V$ as
  required.

  If $g_1(\ttt), \dots, g_r(\ttt)$ are linear forms, then the conditions
  $\epsilon_i g_i(\ttt) > 0$ define open halfspaces in $\RR^s$ bounded by the
  hyperplanes $\{g_i(\ttt)=0\}$ through the origin. Therefore, the interior of
  every connected component of $p(X^\sm(\RR))$ is a cone whose vertex is the
  origin, hence every connected component contains balls of arbitrarily large
  radius, and $C=p(X^\sm(\RR))$.
\end{proof}

\begin{remark}\label{remark:s_ge_2}
  Theorem~\ref{thm:asymptotic} is restricted to $s \ge 2$ due to
  fundamental restrictions in the additive combinatorics
  techniques. Nevertheless, \cite[Theorem~1.1]{arXiv:1307.7641} proves
  a weak approximation result also for $s=1$ by homogenizing the
  equations describing the \emph{vertical} torsors (see \cite[\S
  1.2]{arXiv:1307.7641}), using that a smooth variety satisfies weak
  approximation if and only if the same holds for an open subset. The
  latter is not true for strong approximation, hence we can deduce
  Theorem~\ref{thm:totally_split} only for $s \ge 2$.
\end{remark}

\begin{remark}
  For $X$ as in Theorem~\ref{thm:totally_split}, write
  $P(\ttt)=cg_1(\ttt)^{e_1}\cdots g_r(\ttt)^{e_r}$ as in
  Proposition~\ref{prop:universal_torsors}. Under the assumption that
  $\gcd(e_1, \dots, e_r) = 1$, the Picard group $\Pic(\Xbar^\sm)$ is
  free of finite rank, hence $\Br_1(X^\sm)/\Br_0(X^\sm) \cong
  H^1(k,\Pic(\Xbar^\sm)$ is finite.

  Hence by Remark~\ref{rem:comparison}, $X$ also satisfies central strong
  approximation with algebraic Brauer--Manin obstruction off $\infty$
  under the assumption that $K$ is not totally imaginary or that
  $g_1(\ttt), \dots, g_r(\ttt)$ are linear forms.
\end{remark}

By Remark~\ref{rem:smooth_sa_hp}, Theorem~\ref{thm:totally_split}
implies that $X$ also satisfies the smooth integral Hasse principle
(in the sense of Definition~\ref{def:smooth_sa_hp}). This can be
stated explicitly as follows.

\begin{cor}\label{cor:smooth_hasse_principle}
  Let $P(\ttt) \in \ZZ[t_1, \dots, t_s]$ be a product of linear polynomials
  over $\ZZ$ that are over $\QQ$ pairwise proportional or affinely
  independent. Let $K/\QQ$ be an extension of number fields of degree $n$, and
  let $N_{K/\QQ}(\zz) \in \ZZ[z_1, \dots, z_n]$ be an associated norm form with
  integral coefficients. Suppose that $K$ is not totally imaginary or that the
  factors of $P(\ttt)$ are linear forms. Let $\XX \subset \AAA_\ZZ^{n+s}$ be
  the affine scheme defined by
  \begin{equation*}
    P(\ttt) = N_{K/\QQ}(\zz).
  \end{equation*}
  Let $X = \XX_\QQ$ be the generic fiber.

  If there are $q_v=(\ttt_v,\zz_v) \in \XX(\ZZ_v)$ for all finite
  places $v \in \Omega_\QQ$ and $q_\infty=(\ttt_\infty,\zz_\infty) \in
  X(\RR)$ such that $(q_v)_{v \in \Omega_\QQ}$ is orthogonal to
  $\Br_1(X^\sm)$, with $P(\ttt_v) \in \ZZ_v^\times$ for almost all $v
  < \infty_k$ and $P(\ttt_v) \in k_v^\times$ for all other $v \in
  \Omega_\QQ$, then there are integral points on $\XX$.
\end{cor}

\begin{proof}
  The conditions on $P(\ttt_v)$ ensure that $(q_v) \in U(\AAA_k)
  \subset X^\sm(\AAA_k)$, where $U := X \cap \{P(\ttt) \ne 0\} \subset
  X^\sm$. Hence $(q_v)$ lies in the set
  (\ref{eq:smooth_hp_brauer_set}). As in
  Remark~\ref{rem:smooth_sa_hp}, we obtain an integral point on $\XX$.
\end{proof}

\section{A counterexample}\label{sec:counterexample}

Now we consider the case that $[K:k]=2$ and $P(t)$ is the product of
pairwise distinct linear factors.

\begin{prop}\label{prop:Brauer_quadratic_split}
  Let $k$ be a field of characteristic $0$ with $H^3(k,
  \kbar^\times)=0$. Let $K:=k(\sqrt d)$ be a quadratic extension of
  $k$. Let $P(t)=c(t-a_1)\cdots(t-a_r) \in k[t]$ with pairwise
  distinct $a_1, \dots, a_r \in k$. Let $X$ be defined by (\ref{eq:norm}) with
  $s=1$.

  Then $\Br(X)/\Br_0(X) \cong (\ZZ/2\ZZ)^{r-1}$ is generated by the
  classes of the quaternion algebras $(t-a_i,d) \in \Br(X)$, for $i=1, \dots,
  r-1$.
\end{prop}

\begin{proof}
  Let $U$ be defined again by $P(t) \ne 0$. Analogously to
  (\ref{eq:That_kbar_over_k}), we have the exact sequence
  \begin{equation*}
    0 \to \ZZ \to \ZZ^r \oplus \ZZ[K/k] \to \That \to 0,
  \end{equation*}
  hence $\That \cong \ZZ^{r-1} \oplus \ZZ[K/k]$ as a
  $\Gamma_k$-module. Furthermore, we have $\Mhat \cong \ZZ[K/k]^r$.

  We have
  \begin{equation*}
    H^1(k, \Pic(\Xbar)) \cong H^1(K/k, \Pic(X_K)) \cong H^2(K/k, K[U]^\times/K^\times) \cong (\ZZ/2\ZZ)^{r-1}.
  \end{equation*}
  Indeed, for the first isomorphism, we note that $\Pic(\Xbar)$ is
  torsion-free and split by $K$, for the second isomorphism, we use
  (\ref{eq:seq_That_Mhat_Pic_K}) and $H^i(K/k, \Mhat)=0$ for $i>0$
  since $\Mhat \cong \ZZ[K/k]^r$ is an induced $\Gal(K/k)$-module, and
  for the third isomorphism, we note that $K[U]^\times/K^\times$ is
  $\ZZ^{r-1} \oplus \ZZ[K/k]$ as a $\Gal(K/k)$-module, as in our
  computation of $\That$ above.

  Since $X$ is smooth, Lemma~\ref{lem:Br_formula} implies
  $\Br(X)/\Br_0(X) \cong (\ZZ/2\ZZ)^{r-1}$.

  Now we describe $\Br(X)/\Br_0(X)$ explicitly. Let $\beta_i$ be the
  quaternion algebra $(t-a_i,d)$ over $k(X)$.

  First we show that $\beta_i \in \Br(X)$. Indeed, $(t-a_i,d)$ is clearly
  well-defined on $U_i = \{t \ne a_i\} \subset X$, and
  $(P(t)/(t-a_i),d)$ is well-defined on $U_i' = \bigcap_{j \ne i} \{t
  \ne a_j\}$. Since $U_i \cup U_i' = X$ and
  $(t-a_i,d)=(P(t)/(t-a_i),d)$ in $\Br(k(X))$ (since $P(t) =
  N_{K/k}(\zz)$ implies $(P(t),d)=(N_{K/k}(\zz),d) = 0$), we can
  extend $(t-a_i,d)$ to a well-defined element of $\Br(X)$.

  Next we show that $\beta_1, \dots, \beta_{r-1}$ are $\ZZ/2\ZZ$-linearly independent
  elements in $\Br(k(X))/\Br_0(X)$. Again let $p : X \to \AAA^1_k$ be the
  projection to the $t$-coordinate. This induces a map $\Br(k(t)) \to
  \Br(k(X))$. Let $\eta$ be the generic point of $\AAA^1_k$ and $X_\eta$ the
  generic fiber of $p$. By \cite[Satz, p.~465]{MR1545510}, the kernel of
  $\Br(k(t)) \to \Br(k(X))$ is generated by $(P(t),d)$, hence the induced map
  $\phi: \Br(k(t)) \to \Br(k(X))/\Br_0(X)$ has kernel generated by $(P(t),d)$
  and $\Br(k)$. For $n_1, \dots, n_{r-1} \in \ZZ/2\ZZ$, let $\beta :=
  \sum_{i=1}^{r-1} n_i\beta_i \in \Br(k(t))$. We claim that $\beta \in \ker\phi$ if
  and only if $\beta=0$, which would imply the linear independence of $\beta_1, \dots,
  \beta_{r-1}$ and complete the proof of this theorem.

  Indeed, let $\chi$ be the non-trivial character of $\Gal(K/k)$, and
  let $D_i$ be the divisor $\{t=a_i\}$ of $\AAA^1_k$, for $i=1, \dots,
  r$. By \cite[Proposition~1.1.3]{MR1285781}, the residue of $\beta_i$ on
  $D_i$ is $\chi$, and the residue of $\beta_i$ on any other divisor of
  $\AAA^1_k$ is $0$. Hence the residue of $(P(t),d)$ is $\chi$ on $D_1,
  \dots, D_r$ and $0$ on any other divisor, and any element of $\ker
  \phi$ has the same residue ($\chi$ or $0$) on all of $D_1, \dots,
  D_r$. But since $\beta$ has residue $0$ on $D_r$, it must have residue
  $0$ on $D_1,\dots, D_{r-1}$ as well, which is only possible for
  $n_1=\dots=n_{r-1}=0$, i.e., $\beta=0$.
\end{proof}

The following example illustrates that in Question~\ref{qu:strong}, an
unboundedness condition at an archimedean place is necessary. We
construct a variety over $\QQ$ with one bounded and one unbounded
connected component over $\RR$ and a point orthogonal to the Brauer
group lying on the bounded component that cannot be approximated
arbitrarily close. In particular, it is not enough to require that
$X(k_v)$ is unbounded for one archimedean $v$.

\begin{example}\label{exa:bounded}
  Let $X \subset \AAA^3_\QQ$ be defined by
  \begin{equation}\label{eq:example}
    t(t-2)(t-10)=x^2+y^2.
  \end{equation}
  We note that the projection $p: X \to \AAA^1_\QQ$ to the
  $t$-coordinate has the bounded connected component $[0,2]$ and the
  unbounded connected component $[10,\infty)$. In fact, $X(\RR)$ has
  precisely two connected components, namely $X_1:=p^{-1}([0,2])$ bounded,
  $X_2:=p^{-1}([10,\infty))$ unbounded.

  Our variety $X$ has an adelic point $(q_v) \in X(\RR) \times
  \prod_p X(\ZZ_p)$ given by
  \begin{equation*}
    q_v = (t_v,x_v,y_v) :=
    \begin{cases}
      (5,x_5,y_5), & v = 5,\\
      (1,3,0), & v \ne 5
    \end{cases}
  \end{equation*}
  where $(x_5,y_5) \in \ZZ_5^2$ is a solution of $x_5^2+y_5^2=-75$,
  which exists by Hensel's lemma.

  This point is orthogonal to $\Br(X)$. Indeed, $\Br(X)/\Br_0(X)$ is generated
  by $\beta_1:=(t,-1)$ and $\beta_2:=(t-2,-1)$ by
  Proposition~\ref{prop:Brauer_quadratic_split}. We have
  \begin{align*}
    \sum_{v \in \Omega_\QQ} \beta_i(q_v) &= \sum_{v \ne 5} \inv_v(\beta_i((1,3,0))) + \inv_5(\beta_i((5,x_5,y_5)))\\
    &=\inv_5(\beta_i((1,3,0))) + \inv_5(\beta_i((5,x_5,y_5))) = 0
  \end{align*}
  since $(1,-1)_5=(5,-1)_5=(-1,-1)_5=(3,-1)_5=0$.

  Our goal is to show by contradiction that there is no integral point
  $q=(t,x,y) \in X(\ZZ)$ that is very close to our adelic point
  $(q_v)$ at the places $2$ and $5$. More precisely, we show that
  $|t-1|_2 \le \frac 1 8$ and $|t-5|_5 \le \frac 1{25}$ are
  impossible. This is done via the following claim.

  \textbf{Claim:} Given $(q_v')=(t_v',x_v',y_v') \in (X(\RR) \times
  \prod_p X(\ZZ_p))^{\Br(X)}$ satisfying $|t_2'-1|_2 \le \frac 1 8$, then
  $q_\infty'$ lies in the bounded real component $X_1$.

  This claim leads to our goal as follows. A point $q \in X(\ZZ)$ very
  close to $(q_v)$ at $2$ and $5$ lies on $X_1$ by the claim. By the
  product formula, we get the contradiction
  \begin{equation*}
    1 = \prod_{v \in \Omega_k} |t|_v
    \le |t|_\infty \cdot |t|_5 \le 2 \cdot\frac 1 5 < 1,
  \end{equation*}
  using $t \in \ZZ \setminus \{0\}$ and $|t-5|_5 \le
  \frac{1}{25}$. Hence such a $q \in X(\ZZ)$ does not exist, and our
  adelic point $(q_v)$ cannot be approximated closely.

  It remains to prove the claim. For finite $p \ne 2,5$ and any
  $q_p'=(t_p',x_p',y_p') \in X(\ZZ_p)$, we claim
  \begin{equation}\label{eq:Br_0_not_25}
    \text{$\beta_i(q_p)=0$ in $\Br(\QQ_p)$ for $i=1,2$.}
  \end{equation}
  If the Legendre symbol $\left(\frac{-1}p\right)=1$, then
  $(t_p',-1)=0$ in $\Br(\QQ_p)$ for any $t_p' \in
  \QQ_p^\times$. Therefore, we may assume
  $\left(\frac{-1}p\right)=-1$, hence $\QQ_p(\sqrt{-1})/\QQ_p$ is
  unramified of degree $2$, so $x_p'^2+y_p'^2$ has even valuation at
  $p$. If $v_p(t_p')=0$, then $(t_p',-1)=0$ in $\Br(\QQ_p)$. If
  $v_p(t_p')\ge 1$, then $v_p(t_p'-2)=v_p(t_p'-10)=0$; because of
  (\ref{eq:example}), $v_p(t_p')$ is also even, hence $(t_p',-1)=0$ in
  $\Br(\QQ_p)$. Similarly, we prove $(t_p'-2,-1)=0$.

  For any $(q_v')=(t_v',x_v',y_v') \in (X(\RR) \times \prod_p
  X(\ZZ_p))^{\Br(X)}$ with $|t-1|_2 \le \frac 1 8$, obviously $t \ne
  0$. By (\ref{eq:Br_0_not_25}),
  \begin{align*}
    &\sum_{v \in \Omega_k} \inv_v(\beta_i(q))=\inv_2(\beta_i(q))+\inv_5(\beta_i(q))+\inv_\infty(\beta_i(q))\\
    &=
    \begin{cases}
      (t,-1)_\infty,& i=1,\\
      1+(t-2,-1)_\infty,& i=2.
    \end{cases}
  \end{align*}
  since $(t,-1)_5=(t-2,-1)_5=0$ (because
  $\left(\frac{-1}{5}\right)=1$) and $(t,-1)_2=(1,-1)_2=0$ and
  $(t-2,-1)=(1-2,-1)_2=1$ (because $|t-1|_2 \le \frac 1 8$).

  By global class field theory, $\sum_{v \in \Omega_k} \inv_v(\beta_i(q))
  = 0$, hence $(t,-1)_\infty=0$ and $(t-2,-1)_\infty=1$. Therefore, $0
  \le t \le 2$, hence $q'_\infty \in X_1$, which was our
  claim.
\end{example}

\bibliographystyle{alpha}

\bibliography{norm_strong}

\end{document}